\documentclass{amsart}
\usepackage{graphicx}
\usepackage{amssymb,amscd,amsthm,amsxtra}
\usepackage{latexsym}
\usepackage{epsfig}

\newtheorem{thm}{Theorem}[section]
\newtheorem{cor}[thm]{Corollary}
\newtheorem{lem}[thm]{Lemma}
\newtheorem{prop}[thm]{Proposition}
\theoremstyle{definition}
\newtheorem{defn}[thm]{Definition}
\theoremstyle{remark}
\newtheorem{rem}[thm]{Remark}
\numberwithin{equation}{section}

\newcommand{\R}{\mathbb R}

\newcommand{\eps}{\epsilon}

\newcommand{\p}{\partial}
\newcommand{\di}{\displaystyle}

\newcommand{\comment}[1]{}

\begin{document}

\title[The thin one-phase problem]{$C^{2,\alpha}$ regularity of flat free boundaries for the thin one-phase problem.}
\author{D. De Silva}
\address{Department of Mathematics, Barnard College, Columbia University, New York, NY 10027}
\email{\tt  desilva@math.columbia.edu}
\author{O. Savin}
\address{Department of Mathematics, Columbia University, New York, NY 10027}\email{\tt  savin@math.columbia.edu}


\begin{abstract}We prove $C^{2,\alpha}$ regularity of sufficiently flat free boundaries, for the {\it thin} one-phase problem in which the free boundary occurs on a lower dimensional subspace. This problem appears also as a model of a one-phase free boundary problem in the context of the fractional Laplacian $(-\Delta)^{1/2}$.\end{abstract} 
\maketitle

\section{Introduction}
Let $g(x,s)$ be a continuous non-negative function in the  ball $B_1 \subset \R^{n+1}= \R^n \times \R,$ which vanishes on a subset of $\R^{n} \times\{ 0\}$ and it is even in the $s$ variable. We consider the following free boundary problem 

\begin{equation}\label{FBintro}\begin{cases}
\Delta g = 0, \quad \textrm{in $ B_1^+ 
(g):= B_1 \setminus \{(x,0) : g(x,0) = 0 \} ,$}\\
\dfrac{\p g}{\p U}= 1, \quad \textrm{on  $F(g):=  \p_{\R^n}\{x \in \mathcal{B}_1 : g(x,0)>0\} \cap \mathcal{B}_1,$} 
\end{cases}\end{equation}
where \begin{equation}\label{nabla_U}
\dfrac{\p g}{\p U}(x_0): = \di\lim_{t \rightarrow 0^+} \frac{g(x_0+t\nu(x_0),0)}{\sqrt t},  \quad x_0 \in F(g) \end{equation} with $\nu(x_0)$ the normal to $F(g)$ at $x_0$ pointing toward  $\{x : g(x,0)>0\}$ and  $\mathcal{B}_r \subset \R^n$ the $n$-dimensional ball of radius $r$ (centered at 0).

If $F(g)$ is $C^2$ then it can be shown (see Section 7) that any function $g$ which is harmonic in $B^+_1(g)$ has an asymptotic expansion at a point $ x_0\in F(g),$
$$g(x,s) = \alpha(x_0) U((x-x_0) \cdot \nu(x_0), s)  + o(|x-x_0|^{1/2}+s^{1/2}).$$
Here $U(t,s)$ is the real part of $\sqrt{z}$ which 
in the polar coordinates $$t= r\cos\theta, \quad  s=r\sin\theta, \quad r\geq 0, \quad -\pi \leq  \theta \leq \pi,$$ is given by
\begin{equation}\label{U}U(t,s) = r^{1/2}\cos \frac \theta 2. \end{equation}
Then, the limit in \eqref{nabla_U} represents the coefficient $\alpha(x_0)$ in the expansion above (which justifies our notation)
$$\frac{\p g}{\p U}(x_0)=\alpha(x_0)$$ 
and our free boundary condition requires that $\alpha \equiv 1$ on $F(g).$ 

Solutions to our free boundary problem \eqref{FBintro} are critical points to the energy functional 
$$E(g):= \int_{B_1}  |\nabla g|^2 \,dx\,ds + \frac \pi 2 \mathcal H^{n}(\{g>0\} \cap \mathcal B_1). $$ If the second term is replaced by $\mathcal H^{n+1}(\{g>0\})$, we obtain 
the classical one-phase free boundary problem (see for example \cite{AC}.) In our case the free boundary occurs on the lower dimensional subspace $\R^{n} \times \{0\}$ and for this reason  
we  refer to \eqref{FBintro} as to the {\it thin one-phase} free boundary problem.

This free boundary problem was first considered by Caffarelli, Roquejoffre and Sire \cite{CafRS} as a model of a one-phase Bernoulli type free boundary problem in the context of the fractional Laplacian. It is relevant in applications when turbulence or long-range interactions are present, for example in flame propagation and also in the propagation of surfaces of discontinuities. For further information on this model see \cite{CafRS} and the references therein.

In this paper we are interested in the question of regularity for the free boundary $F(g)$. Concerning this issue the authors of \cite{CafRS} proved that in dimension $n=2$, a Lipschitz free boundary is $C^1$. In \cite{DR}, the first author and Roquejoffre showed that in any dimension if the free boundary $F(g)$ is sufficiently flat then it is $C^{1,\alpha}.$

This paper is the first of a series of papers, which investigate the regularity of $F(g)$ and in particular the question of whether Lipschitz free boundaries are smooth.  This basic question was answered positively in the case of minimal surfaces by De Giorgi \cite{DG} and by Caffarelli \cite{C1} for the standard one-phase free boundary problem. 

Our strategy to obtain the regularity of Lipschitz free boundaries is to use a Weiss-type monotonicity formula \cite{W} combined with flatness results and ad hoc Schauder type estimates near the free boundary. To implement this method we need to obtain first $C^{2,\alpha}$ estimates for flat free boundaries, which we achieve in this paper.  Unlike the case of minimal surfaces and of the standard one-phase problem, $C^{2,\alpha}$ estimates do not seem to follow  easily from $C^{1,\alpha}$. 
It appears that $C^{2,\alpha}$  is the critical regularity needed to obtain $C^\infty$ smoothness of the free boundary, as well as the regularity needed to implement our blow-up analysis.

 The following is the main result of this paper  (see Section 2 for the precise definition of viscosity solution to \eqref{FBintro}).
 
\begin{thm} \label{mainT}There exists $\bar \eps >0$ small depending only on $n$, such that if $g$ is a viscosity solution to \eqref{FBintro}  satisfying
\begin{equation*} \{x \in \mathcal{B}_1 : x_n \leq - \bar \eps\} \subset \{x \in \mathcal{B}_1 : g(x,0)=0\} \subset \{x \in \mathcal{B}_1 : x_n \leq \bar \eps \},\end{equation*} then $F(g)$ is a $C^{2,\alpha}$ graph in $\mathcal{B}_{\frac 1 2}$ for every $\alpha \in (0,1)$ with  $C^{2,\alpha}$ norm bounded by a constant depending on $\alpha$ and $n$. \end{thm}

The proof of Theorem \ref{mainT} follows the lines of the flatness theorem in \cite{DR}, which is inspired by the regularity theory developed by the second author in \cite{S}. In this case the proof is more technical since we need to approximate the free boundary quadratically. To do so, we introduce a family of approximate solutions $V_{\mathcal S,a,b}$ which have the same role as quadratic polynomials in the regularity theory of elliptic equations. Such family will be used also in a subsequent paper to obtain boundary Schauder type estimates for solutions to our problem. 

In the last section of this paper we also prove some useful  general facts about viscosity solutions $g$ to our free boundary problem \eqref{FBintro}, such as $C^{1/2}$-optimal regularity, asymptotic expansion near regular points of the free boundary and compactness.

The paper is organized as follows. In Section 2 we recall notation, definitions and some basic  results from \cite{DR}, including the linearized problem associated to \eqref{FBintro}. Section 3 is devoted to the construction of the quadratic approximate solutions $V_{\mathcal S, a,b}$.  In Section 4 we prove a Harnack type inequality for solutions to \eqref{FBintro}. In Section  5 we establish the improvement of flatness result via a compactness argument which makes crucial use of the Harnack inequality of Section 4. Our argument reduces the problem to studying the regularity of solutions to  the linearized problem. This is pursued in Section 6. We conclude the paper with Section 7 where we provide some general facts about viscosity solutions to \eqref{FBintro}.
\section{Definitions and basic lemmas}

In this section we recall notation, definitions and some necessary results from \cite{DR}. 

\subsection{Basic facts.} Throughout the paper, constants which depend only on the dimension $n$ will be called universal.  In general, small constants will be denoted by $c,c_i$ and large constants by $C, C_i$ and they may change from line to line in the body of the proofs. 

A point $X \in \R^{n+1}$ will be denoted by $X= (x,s) \in \R^n \times \R$, and sometimes $x=(x',x_n)$ with $x'=(x_1,\ldots, x_{n-1}).$ 

A ball in $\R^{n+1}$ with radius $r$ and center $X$ is denoted by $B_r(X)$ and for simplicity $B_r = B_r(0)$. Also $\mathcal{B}_r$ denotes the $n$-dimensional ball $B_r \cap \{s=0\}$. 

Let $v \in C(B_1)$ be a non-negative function. We associate to $v$ the following sets: \begin{align*}
& B_1^+(v) := B_1 \setminus \{(x,0) : v(x,0) = 0 \} \subset \R^{n+1};\\
& \mathcal{B}_1^+(v):= B_1^+(v) \cap \mathcal{B}_1 \subset \R^{n};\\
& F(v) := \p_{\R^n}\mathcal{B}_1^+(v)\cap \mathcal{B}_1 \subset \R^{n}.
\end{align*}  Often subsets of $\R^n$ are embedded in $\R^{n+1}$, as it will be clear from the context. 

We consider the thin one-phase free boundary problem

\begin{equation}\label{FB}\begin{cases}
\Delta g = 0, \quad \textrm{in $B_1^+(g) ,$}\\
\dfrac{\p g}{\p U}= 1, \quad \textrm{on $F(g)$}, 
\end{cases}\end{equation}
where 
$$\dfrac{\p g}{\p U}(x_0):=\di\lim_{t \rightarrow 0^+} \frac{g(x_0+t\nu(x_0),0)} {\sqrt t} , \quad \textrm{$X_0=(x_0,0) \in F(g)$}.$$ 
Here $\nu(x_0)$ denotes the unit normal to $F(g)$, the free boundary of $g$, at $x_0$ pointing toward $\mathcal{B}_1^+(g)$.




We now recall the notion of viscosity solutions to \eqref{FB}, introduced in \cite{DR}. 

\begin{defn}Given $g, v$ continuous, we say that $v$
touches $g$ by below (resp. above) at $X_0 \in B_1$ if $g(X_0)=
v(X_0),$ and
$$g(X) \geq v(X) \quad (\text{resp. $g(X) \leq
v(X)$}) \quad \text{in a neighborhood $O$ of $X_0$.}$$ If
this inequality is strict in $O \setminus \{X_0\}$, we say that
$v$ touches $g$ strictly by below (resp. above).
\end{defn}

\begin{defn}\label{defsub} We say that $v \in C(B_1)$ is a (strict) comparison subsolution to \eqref{FB} if $v$ is a  non-negative function in $B_1$ which is even with respect to $s=0$ and it satisfies
\begin{enumerate} \item $v$ is $C^2$ and $\Delta v \geq 0$ \quad in $B_1^+(v)$;\\
\item $F(v)$ is $C^2$ and if $x_0 \in F(v)$ we have
$$v (x_0+t\nu(x_0),0) = \alpha(x_0) \sqrt t + o(\sqrt t), \quad \textrm{as $t \rightarrow 0^+,$}$$ with $$\alpha(x_0) \geq 1,$$ where $\nu(x_0)$ denotes the unit normal at $x_0$ to $F(v)$ pointing toward $\mathcal{B}_1^+(v);$\\
\item Either $v$ is not harmonic in $B_1^+(v)$ or $\alpha(x_0) >1$ at all $x_0 \in F(v).$
\end{enumerate} 
\end{defn}

Similarly one can define a (strict) comparison supersolution. 

\begin{defn}We say that $g$ is a viscosity solution to \eqref{FB} if $g$ is a  continuous non-negative function in $B_1$ which is even with respect to $s=0$ and it satisfies
\begin{enumerate} \item $\Delta g = 0$ \quad in $B_1^+(g)$;\\ \item Any (strict) comparison subsolution (resp. supersolution) cannot touch $g$ by below (resp. by above) at a point $X_0 = (x_0,0)\in F(g). $\end{enumerate}\end{defn}

\begin{rem}\label{rescale} We remark that if $g$ is a viscosity solution to \eqref{FB} in $B_\lambda$, then $$g_{\lambda}(X) = \lambda^{-1/2} g(\lambda X), \quad X \in B_1$$ is a viscosity solution to \eqref{FB} in $B_1.$
\end{rem}

Finally, we state for completeness the boundary Harnack inequality which will be often used throughout the paper.  This version follows from the boundary Harnack inequality  proved in \cite{CFMS}.

\begin{thm}[Boundary Harnack Inequality] Let $v$ be harmonic in $B_1^+(v)$ and let $F(v)$ be a Lipschitz graph in the $e_n$-direction (pointing towards the positive phase) with $0\in F(v)$. If $w$ is harmonic in $B_1^+(w)=B_1^+(v),$ then 
$$\frac{w}{v} \leq C \frac w v (\frac 1 2 e_n) \quad \mbox{in $B_{3/4}$,}$$ with $C$ depending only on $n$ and on the Lipschitz constant of $F(v).$
\end{thm}  

\subsection{The function $\tilde g$} 

Here and henceforth we denote by $P$ the half-hyperplane $$P:= \{X \in \R^{n+1} : x_n \leq 0, s=0\}$$ and by $$L:= \{X \in \R^{n+1}: x_n=0, s=0\}.$$ 
 Also, throughout the paper
we call $U(X):=U(x_n,s),$ where $U$ is the function defined in \eqref{U}.

Let $g$ be a  continuous non-negative function in $\overline{B}_\rho$. As in \cite{DR}, we define the multivalued map $\tilde g$ which associate to each $X \in \R^{n+1} \setminus P$ the set $\tilde g(X) \subset \R$ via the formula 
\begin{equation}\label{deftilde} U(X) = g(X - w e_n), \quad \forall w \in \tilde g(X).\end{equation} 
We write $ \tilde g(X)$ to denote any of the values in this set.

This change of variables has the same role as the partial Hodograph transform for the standard one-phase problem. Our free boundary problem becomes a problem with fixed boundary for $\tilde g$,  and the limiting values of $\tilde g$ on $L$ give the free boundary of $g$ as a graph in the $e_n$ direction. 

Recall that if g satisfies the $\eps$-flatness assumption \begin{equation}\label{flattilde}U(X - \eps e_n) \leq g(X) \leq U(X+\eps e_n) \quad \textrm{in $B_\rho,$ for $\eps>0$}\end{equation} then $\tilde g(X) \neq \emptyset$ for $X \in B_{\rho-\eps} \setminus P$  and $|\tilde g(X)| \leq \eps,$ hence  we can associate to $g$ a possibly multi-valued function $\tilde{g}$ defined at least on $B_{\rho-\eps} \setminus P$ and taking values in $[-\eps,\eps]$ which satisfies\begin{equation}\label{til} U(X) = g(X - \tilde g(X) e_n).\end{equation} 
 Moreover if $g$ is strictly monotone  in the $e_n$-direction in $B^+_\rho(g)$, then $\tilde{g}$ is single-valued. 

We recall the following lemmas  from \cite{DR}.

\begin{lem}\label{elem} Let $g, v$ be non-negative continuous functions in $B_\lambda$ with  $v$ strictly increasing in the $e_n$-direction in $B_\lambda^+(v).$  Assume that $g$ and $v$ satisfy the flatness assumption \eqref{flattilde} in $B_\lambda$
for $\eps>0$ small. If 
$$v \leq g \quad \text{in $B_\lambda,$}$$ then
$$\tilde v \leq \tilde g \quad \text{on  $B_{\lambda-\eps} \setminus P.$}$$
Viceversa, if $$\tilde v \leq \tilde g \quad \text{on $B_\sigma \setminus P,$}$$ for some $0< \sigma< \lambda - \eps,$ then
$$v \leq g \quad \text{on $B_{\sigma-\eps}$}.$$
\end{lem}




\begin{lem}\label{linearcomp}
Let $g, v$ be respectively a solution and a subsolution to \eqref{FB} in $B_2$, with $v$ strictly increasing in the $e_n$-direction in $B_2^+(v).$ Assume that $g$ and $v$ satisfy the flatness assumption \eqref{flattilde} in $B_2$
for $\eps>0$ small. If,
\begin{equation}\label{start}
 \tilde v + \sigma \leq  \tilde g \quad \text{in $(B_{3/2} \setminus \overline{B}_{1/2}) \setminus P,$} 
\end{equation} for some $\sigma >0,$ then 
\begin{equation}\label{conclusion}
 \tilde v+ \sigma \leq \tilde g  \quad \text{in $B_{3/2} \setminus P.$} 
\end{equation} 
\end{lem}

Finally, given a Lipschitz function $\phi$ defined on $B_{\lambda}(\bar X)$,  with values in $[-1,1]$, then  for all $\eps>0$ small there exists a unique function $\varphi_\eps$ defined at least on $B_{\lambda-\eps}(\bar X)$ such that \begin{equation}\label{deftilde2} U(X) = \varphi_\eps(X - \eps \phi (X)e_n), \quad X \in B_\lambda(\bar X), \end{equation} that is $$\tilde{\varphi_\eps}= \eps\phi.$$
Moreover such function $\varphi_\eps$ is increasing in the $e_n$-direction. 

If 
 $g$ satisfies the flatness assumption \eqref{flattilde} in $B_1$ and $\phi$ is as above then (say $\lambda<1/4$, $\bar X \in B_{1/2},$) 
\begin{equation}\label{gtildeg}\tilde \varphi_\eps \leq \tilde g \quad \text{in $B_\lambda(\bar X) \setminus P$} \Rightarrow \varphi_\eps \leq g \quad \text{in $B_{\lambda -\eps}(\bar X)$}.\end{equation}

The following Proposition will be used in the compactness argument for the proof of the improvement of flatness in Section 6.

\begin{prop}\label{1}Let $\phi$ be a smooth function in $B_\lambda(\bar X)\subset \R^{n+1} \setminus P$. Define (for $\eps>0$ small) the function $\varphi_\eps$ as above by\begin{equation}\label{deftilde3} U(X) = \varphi_\eps(X -  \eps\phi(X)e_n).\end{equation}Then, \begin{equation}\label{laplaceest} \Delta \varphi_\eps = \eps \Delta(U_n \phi)+ O(\eps^2),\quad \textrm{in $B_{\lambda/2}(\bar X)$}\end{equation} with the function in $O(\eps^2)$ depending on $\|\phi\|_{C^5}$ and $\lambda$.\end{prop}

\begin{proof} For notional simplicity we drop the subindex $\eps$ in the definition of $\varphi_\eps.$
From formula \eqref{deftilde3} and Taylor's theorem, we have that 
\begin{equation}\label{Taylor}U(X) =\varphi(X) - \eps \varphi_n(X)\phi(X) +\eps^2\Psi(X), \quad \textrm{in $B_{\lambda/2}(\bar X)$}
\end{equation}
with $\|\Psi\|_{C^3(B_{\lambda/2}(\bar X))} \leq C$
and $C$ depending on $\|\phi\|_{C^5}$ and $\lambda$. Thus, $$U_n(X) = \varphi_n(X) + O(\eps).$$ Combining this formula for $\varphi_n(X)$ and \eqref{Taylor} we obtain
\begin{equation*}U(X) =\varphi(X) - \eps U_n(X)\phi(X) +O(\eps^2).\end{equation*}
Hence, using that $U$ is harmonic,
\begin{equation*}0 = \Delta U(X) =\Delta \varphi(X) - \Delta(\eps U_n\phi)(X) +O(\eps^2),\end{equation*} as desired.
\end{proof}

We remark that in fact the function in $O(\eps^2)$ only depends on $\lambda$ if we choose $\eps$ small enough depending on $\|\phi\|_{C^5}$.

\subsection{The linearized problem.}  

We recall here the linearized problem associated to \eqref{FB}.  Here and later $U_n$ denotes the $x_n$-derivative of the function $U$. Recall that  $$P:= \{X \in \R^{n+1} : x_n \leq 0, s=0\}, \quad L:= \{X \in \R^{n+1}: x_n=0, s=0\}.$$



Given  $h \in C(B_1)$  and  $X_0=(x'_0,0,0) \in B_1 \cap L,$ we call
$$|\nabla_r h |(X_0) := \di\lim_{(x_n,s)\rightarrow (0,0)} \frac{h(x'_0,x_n, s) - h (x'_0,0,0)}{r}, \quad   r^2=x_n^2+s^2 .$$
Once the change of unknowns  \eqref{deftilde} has been done, the linearized problem associated to \eqref{FB} is 
\begin{equation}\label{linear}\begin{cases} \Delta (U_n h) = 0, \quad \text{in $B_1 \setminus P,$}\\ |\nabla_r h|=0, \quad \text{on $B_1\cap L$.}\end{cases}\end{equation}




\begin{defn}\label{linearsol}We say that $h$ is a solution to \eqref{linear}  if $h \in C(B_1)$, $h$ is even with respect to $\{s=0\}$ and it satisfies
\begin{enumerate}\item $\Delta (U_n h) = 0$ \quad in $B_1 \setminus P$;\\ \item $h$ cannot be touched by below (resp. by above) at any  $X_0=(x'_0,0,0) \in B_1 \cap L$, by a continuous function $\phi$ which satisfy $$\phi(X) = \phi(X_0) + a(X_0)\cdot (x' - x'_0) + b(X_0) r + O(|x'-x'_0|^2 + r^{3/2}), $$ with $b(X_0) >0$ (resp. $b(x_0) <0$). \end{enumerate}\end{defn}

In Section 6, we will prove a quadratic expansion for solutions to the linearized problem which yields the following corollary.

\begin{cor}\label{classnewcor} Let  $h$ be a solution to \eqref{linear} such that $|h|  \leq 1$. Given any $\alpha \in (0,1)$, there exists $\eta_0$ depending on $\alpha$ , such that $h$ satisfies
\begin{equation*}|h(X) - (h(0) +  \xi_0 \cdot x' + \frac 1 2 (x')^TM_0x'  - \frac{a_0}{2}r^2 -b_0rx_n)| \leq \frac 1 4 \eta_0^{2+\alpha} \quad \textrm{in $B_{\eta_0},$}\end{equation*}  with $r^2=x_n^2+ s^2$, for some $a_0,b_0 \in \R, \xi_0 \in \R^{n-1}, M_0 \in S^{(n-1) \times (n-1)}$ with $$|\xi_0|, |a_0|,  |b_0|, \|M_0\| \leq C, \quad \textrm{$C$ universal}$$ and $$ a_0 +b_0 -tr M_0=0.$$ \end{cor}







\section{A family of functions.}

In this section we introduce a family of functions $V_{\mathcal S, a, b}$ which approximate our solution quadratically. These functions will be often used as comparison subsolutions/supersolutions.
We establish here some of their basic properties, including their behavior under the change of coordinates $V \rightarrow \tilde V$ (see Proposition \ref{estimate}).

 We start by presenting some basic properties of the solution $U$ defined in the introduction. Recall that $$U(t,s) := \rho^{1/2}\cos \frac{\beta}{2},$$
  where 
$$t= \rho\cos\beta, \quad  s=\rho\sin\beta, \quad \rho\geq 0, \quad -\pi \leq  \beta \leq \pi.$$ 
 
 We will use the following properties of the function $U$:
 \begin{enumerate}
 \item $\Delta U = 0, \quad U>0 \quad \textrm{in $\R^{n+1} \setminus P.$}$ \\
 \item $U_t= \frac 1 2  \rho^{-1/2}\cos \frac{\beta}{2}= \dfrac {1}{2\rho} U$ and $U_t > 0$ in $\R^{n+1} \setminus P.$\\
 \end{enumerate}

 Since $U_t$ is positive harmonic in $\R ^2 \setminus \{(t,0), \quad t \le 0 \} $, homogenous of
degree $-1/2$ and vanishes continuously on $\{(t,0), \quad t < 0 \} $ one can see from boundary Harnack inequality (or by direct computation) that values of $U_t$ at nearby points with the same second coordinate are comparable in diadic rings. Precisely we have
\begin{equation}\label{diadic}\frac{U_t(t_1,s)}{U_t(t_2,s)} \le C \quad \textrm{if} \quad |t_1-t_2| \le \frac 1 2 |(t_2,s)|.\end{equation}

Next we introduce the family $V_{\mathcal S, a,b}$. For any $a,b \in \R$ we define the following family of (two-dimensional) functions (given in polar coordinates $(\rho, \beta)$)
\begin{equation}
v_{a,b}(t,s):= (1+\frac{a}{4}\rho+\frac{b}{2}t)\rho^{1/2}\cos \frac{\beta}{2},
\end{equation}
that is 
 $$v_{a,b}(t,s) = (1+\frac{a}{4}\rho+\frac{b}{2}t)U(t,s) = U(t,s) + o(\rho^{1/2}).$$

Given a surface $\mathcal S= \{x_n = h(x')\} \subset \R^n$,  we call $\mathcal P_{\mathcal S,X}$ the 2D plane passing through $X=(x,s)$ and perpendicular to $\mathcal S$,
that is the plane containing $X$ and generated by the $s$-direction and the normal direction from $(x,0)$ to $\mathcal S$.

We define the family of functions 
\begin{equation}\label{vS}
V_{\mathcal S, a,b} (X): = v_{a,b}(t,s), \quad X=(x,s),
\end{equation}
with $t=\rho\cos \beta, s=\rho\sin \beta$ respectively the first and second coordinate of $X$ in the plane $\mathcal P_{\mathcal S,X}$. In other words, $t$ is the signed distance from $x$ to $\mathcal S$ (positive above $\mathcal S$ in the $x_n$-direction.)

If $$\mathcal S:= \{ x_n = \frac 1 2 (x')^T M x' + \xi' \cdot x'\},$$ for some $M \in S^{(n-1) \times (n-1)}, \xi' \in \R^{n-1}$  we use the notation \begin{equation}\label{vM}
 V_{M,\xi', a,b}(X):= V_{\mathcal S, a,b} (X). \end{equation}
This will be the case  throughout most of the paper.

\begin{defn}\label{assonV} For $\delta >0$ small, we define the following classes of functions $$\mathcal{V}_\delta: = \{V_{M,\xi', a,b} : \ 
\|M\|, |\xi'|, |a|, |b| \leq \delta\},$$ and $$\mathcal V_\delta^0:= \{V_{M,\xi', a,b} \in \mathcal V_\delta : \  a+b- tr M=0\}.$$\end{defn} Most of the times we will work with functions in the class $\mathcal{V}_\delta$, since we deal with the flat  case. 
Notice that if we rescale $V = V_{M,\xi', a. b}$ that is $$V_\lambda(X)=\lambda^{-1/2}V(\lambda X), \quad X\in B_1, $$ then it easily follows  from our definition that $$V_\lambda= V_{\lambda M,\xi', \lambda a, \lambda b}.$$

In the next proposition we provide a condition for a function  $V \in \mathcal V_\delta$ to be a subsolution/supersolution.

\begin{prop}\label{sub} Let $V=V_{M,\xi',a,b} \in \mathcal V_\delta,$  with $\delta \leq \delta_0$ universal. There exists a universal constant $C_0>0$ such that if \begin{equation}\label{Vsub}a+b- tr M \geq C_0 \delta^2\end{equation} then  $V$ is a comparison subsolution to \eqref{FB} in $B_2$. 
\end{prop}

\begin{proof} Clearly from our formula for $v_{a,b}$ the function $V$ satisfies the free boundary condition of Definition \ref{defsub} with $\alpha(x_0) \equiv 1$. We need to check that $\Delta V(X) >0$ at all  $X \in B_2^+(V).$ 

Since that $V(X)$ depends only on $(t,s)$ and 
$$\Delta_x t= -\kappa(x)$$
where $\kappa(x)$ is the sum of the principal curvatures of the parallel surface to $S$ (in $\R^n$) 
passing through $x,$
 we compute that 
\begin{equation}\label{deltaV}\Delta V(X) = \Delta_{(t,s)} v_{a,b} - (\p_t v_{a,b}) \kappa(x).\end{equation}

From our formula for $v_{a,b}$, using polar coordinates we get that\begin{equation}\label{deltav}\Delta_{(t,s)} v_{a,b} = \frac 1 2 (a+b) \rho^{-1/2}\cos \frac{\beta}{2} = (a+b)U_t. \end{equation} 

Also, since $\rho \leq 2,$
\begin{equation}\label{ptv}|\p_t v_{a,b} -  U_t| \leq (|a| + |b|) \rho^{1/2}\cos \frac{\beta}{2} \leq 8 \delta  U_t.\end{equation}
Finally we use that $\kappa_i(x)$ the principal curvatures at $x$ are given by,
\begin{equation}\label{kappa}\kappa_i(x) = \frac{\kappa_i(x^*)}{1-t\kappa_i(x^*)}\end{equation} where $x^*$ is the projection of $x$ onto $S$.  Since $|\xi'|, \|M\| \leq \delta$ we obtain that $$|\kappa_i(x^*)| \leq C\delta, \quad 
|\kappa(x^*) - tr M| \leq C \delta^3$$ for $C$ universal, which in view of \eqref{kappa} give
\begin{equation}\label{curvatures} |\kappa(x) - tr M| \leq C \delta^2.\end{equation}
From \eqref{deltaV} combined with \eqref{deltav}, \eqref{ptv} and \eqref{curvatures} we get that
\begin{equation}\label{need}|\Delta V(X) - (a+b-tr M)U_t | \leq \frac 1 2 C_0\delta^2 U_t \end{equation} for a $C_0$ universal.
It follows that if $$a+b -tr M \geq  C_0 \delta^2$$ then $\Delta V(X) > 0$ as desired.

\end{proof}

Next, we estimate $V_n$ and $\Delta V$ outside a small cone with axis $L$. 

\begin{prop}Let  $V= V_{M,\xi', a, b} \in \mathcal V_\delta$ with $\delta \leq \delta_0$ universal, then
\begin{equation} \label{concl1}c \leq \frac{V_n}{U_n}\leq C, \quad \textrm{in $B_2 \setminus (P\cup \{|(x_n,s)| \leq 10\delta|x'|\}).$}\end{equation}
If $V \in \mathcal V_\delta^0$ then  \begin{equation}\label{estun}|\Delta V(X)| \leq C\delta^2U_n(X) \quad \textrm{in $B_2 \setminus (P\cup \{|(x_n,s)| \leq 10\delta|x'|\}).$}\end{equation} 
\end{prop}

\begin{proof} From our formula
$$V_n(X)=\p_t v_{a,b}(t,s) \frac{\p t}{\p x_n}$$ where $t$ represents the signed distance from $x$ to $\mathcal S.$ Since $\nabla_x t$ is the unit vector at $x$ that has the direction of the normal from $x$ to $\mathcal S$, it makes an angle of order $\delta$ with respect to $e_n$. Hence since $$\frac{\p t}{\p x_n} = \nabla_x t \cdot e_n$$ we get 
\begin{equation}\label{ptxn}
1 \geq \frac{\p t }{\p x_n} \geq  1- C \delta^2   
\end{equation}
 and we obtain
$$\p_t v_{a,b}(t,s)\ge V_n(X) \ge \frac 12 \p_t v_{a,b}(t,s).$$ 
From \eqref{ptv} we see that $\p_t v_{a,b} \sim U_t$
and we obtain that \begin{equation}\label{utcovt}
2 \p_t U(t,s) \ge V_n(X)\ge \frac 1 4 \p_t U(t,s).
\end{equation}
Thus to obtain our claim we need to replace $t$ with $x_n$ in the inequality above.

Since in $B_{2|x|}$ the surface $\mathcal S$ is in a $4 \delta|x|$ neighborhood of $x_n=0$ we find that
$|t-x_n| \le 4 \delta|x|$. If $X$  belongs to the domain in \eqref{concl1} then $$|(x_n,s)| \ge 8 \delta|x| \geq 2 |t-x_n| $$ and we obtain from \eqref{diadic}
\begin{equation}\label{cvu}c \le \frac{U_t(t,s)}{U_t(x_n,s)} \le C\end{equation} which together with \eqref{utcovt} gives the desired conclusion \eqref{concl1}.  

Now \eqref{estun} follows immediately.  Indeed by formula \eqref{need} we have that 
$$|\Delta V(X)| \leq C \delta^2 U_t(t,s),$$
which combined with \eqref{cvu} gives the desired bound.
\end{proof}

\begin{rem}\label{remark}We remark that if $V \in \mathcal V_\delta^0$, then the rescaling $V_\lambda (X)= \lambda^{-1/2} V(\lambda X)$ with $\lambda \le 1$, satisfies 
 \begin{equation}\label{estun2} c \leq \frac{(V_\lambda)_n}{U_n} \leq C,  \quad |\Delta V_\lambda(X)| \leq C\delta^2U_n(X), \end{equation} in the dilation ball of factor $1/\lambda$ $$B_{2/\lambda} \setminus (P\cup \{|(x_n,s)| \leq 10\delta|x'|\}).$$

Indeed $$\Delta V_\lambda(X)=\lambda^{3/2} \Delta V(\lambda X), \quad U_n(X) = \lambda^{1/2} U_n(\lambda X), \quad (V_\lambda)_n(X) = \lambda^{1/2} V_n(\lambda X).$$
\end{rem}

Now we study the behavior of $V \in \mathcal V_\delta$ under the transformation $V \rightarrow \tilde V.$ This will be quite useful in the rest of the paper.

\begin{prop}\label{estimate} Let $V=V_{M,\xi',a,b} \in \mathcal V_\delta$, with $\delta \leq \delta_0$ universal. Then $V$ is strictly monotone increasing in the $e_n$-direction in $B_2^+(V)$. 
Moreover, $\tilde V$ satisfies the following estimate in $B_{2} \setminus P$ 
\begin{equation*}
|\tilde V(X) - \gamma_V (X)| \leq C_1 \delta^2, \quad \gamma_V (X)= \frac{a}{2}r^2+brx_n-\frac 1 2 (x')^T M x' -\xi' \cdot x'
\end{equation*} with  $r=\sqrt{x_n^2+s^2}$ and $C_1$ a universal constant.
\end{prop}

\begin{proof} 
First we show that $v_{a,b}$ satisfies 
\begin{equation}\label{gab}
U(t+\gamma_{a,b}-C\delta^2, s) \le v_{a,b}(t,s) 
\le U(t+\gamma_{a,b}+C\delta^2, s)
\end{equation} where $\rho^2=t^2+s^2$ and $\gamma_{a,b}$ is the following expression depending on $t$ and $s$:
$$\gamma_{a,b}(t,s):=\frac a 2 \rho^2 + b\rho t.$$
Indeed since (see properties of $U$ listed at the beginning of this action) $$|U_{tt}| \le C \rho^{-1} U_t$$we have that if $|\mu| \le \rho/2$ then
$$|U(t+\mu,s)-(U(t,s)+\mu U_t(t,s))| \le \mu^2 |U_{tt}(t',s)| \le C \mu^2 \rho^{-1} U_t(t,s) ,$$
where in the last inequality we used \eqref{diadic}. Thus, since $U_t=U/(2 \rho)$,
$$(1+\frac {\mu}{2 \rho}+ C \frac{\mu^2}{\rho^2})U(t,s) \ge U(t+\mu,s) \ge (1+\frac {\mu}{2 \rho} - C \frac{\mu^2}{\rho^2})U(t,s) .$$
Choosing $$\mu = \tilde \mu \pm 4C \frac{\tilde \mu^2}{\rho}$$ we obtain that 
$$U(t+\tilde \mu + 4C \frac {\tilde \mu^2}{\rho},s) \ge (1+\frac{\tilde \mu}{ 2 \rho})U(t,s) \ge U(t+\tilde \mu - 4C \frac {\tilde \mu^2}{\rho},s), $$  provided that $|\tilde \mu / \rho|<c$, with $c$ sufficiently small.
Since $$v_{a,b}=(1+\frac a 4 \rho + \frac b 2 t)\,U$$ we can apply the inequality above with
$$\tilde \mu = \frac{a}{2} \rho^2+ b t \rho, $$ hence $|\tilde \mu|/\rho \le C \delta$ and obtain the claim.

When $t$ is the signed distance from $x$ to the surface $\mathcal S$ we have  $$t=0 \quad \mbox {on} \quad \mathcal S:=\{x_n=h(x') :=\frac 1 2 x'^T M x' +\xi' \cdot x'\}$$ and by \eqref{ptxn}
$$1 \ge \frac{\p t}{\p x_n} \ge 1-C \delta^2  \quad \mbox{in $B_1$},$$
thus, by integrating this inequality on the segment $(x', h(x')), (x', x_n)$ we get
$$ |t-(x_n-h(x'))| \le C \delta^2.$$
Since in $B_1$, the surface $\mathcal S$ and $x_n=0$ are within distance $\delta$ from each other we have $|t-x_n| \le C \delta$ and hence  $$|\gamma_{a,b}(t,s)-\gamma_{a,b}(x_n,s)| \le \|\nabla v_{a,b}\|_{L^{\infty}} |t-x_n| \leq C \delta^2.$$
From the last two inequalities we have that 
$$|(t+\gamma_{a,b}(t,s)) - (x_n + \gamma_V(X))| \leq C\delta^2,$$
with $$\gamma_V(X)=\gamma_{a,b}(x_n,s)-\frac 1 2 x'^T M x' -\xi' \cdot x'.$$

Using this fact and  \eqref{gab} (and the monotonicity of $U$ in the $e_n$ direction) we obtain  
$$U(X+(\gamma_V(X)-C\delta^2)e_n) \le V(X) \le  U(X+(\gamma_V(X)+C\delta^2)e_n), $$
and the estimate for $\tilde V$ is proved.

Finally, we remark that the monotonicity of $V$ follows from \eqref{utcovt}.
\end{proof}

\begin{rem}\label{Visflat} Notice that from the last inequality in the proof above, we obtain that if $V \in \mathcal{V}_\delta$, then $V$ satisfies the $4\delta$-flatness assumption in $B_1$ (see also \eqref{flattilde}):
$$U(X - 4\delta e_n) \leq V(X) \leq U(X+4\delta e_n).$$ This could be also checked easily directly from the definition of $V.$
\end{rem}

We conclude this section with by  comparing the functions $V$ corresponding to two nearby surfaces.

\begin{lem}\label{distance} Let $\mathcal S_i, i=1,2$ be surfaces with curvature bounded by 2. Let  $$V_i=V_{\mathcal S_i,a_i,b_i} ,  \quad \text{$|a_i|, |b_i| \leq 2$}, \quad i=1,2.$$
Assume that, $$\mathcal S_i\cap B_{2\sigma}=\{x_n=h_i(x')\}, \quad \sigma \leq c$$ with $h_i$ Lipschitz graphs, $h_i(0)=0$,  $|\nabla h_i| \le 1$ and $c$ universal. 
If  $$|a_1-a_2|, |b_1-b_2|\le \eps,  \|h_1-h_2\|_{L^\infty}\le \eps \sigma^2,$$ for some small $\eps \leq c$, 
then $$V_1 (X) \le V_2(X+C \eps \sigma^2e_n)\quad \mbox{in $B_\sigma$} .$$
\end{lem}

\begin{proof} After a rescaling of factor $1/\sigma$, we need to prove our lemma for $\sigma=1$ and with the curvature of $S_i$, $a_i, b_i$ and $\eps$ smaller than $c$ universal.

First we prove that  for $0<\lambda \leq 1$, $$v_{a_1,b_1}(t,s) \le v_{a_2,b_2}(t+C \eps \lambda^2 , s), \quad \lambda \leq \rho=|(t,s)| \leq 2\lambda.$$

By \eqref{utcovt},  $\p_t v_{a,b}$ is proportional to $\p _t U$ in the disk of radius 2. Since on the segment with endpoints $(t,s)$ and $(t+ C \eps \lambda^2,s)$ all the values of $\p_t U$ are comparable (see \eqref{diadic}) we obtain (using $2\rho U_t= U$)
\begin{align*}
v_{a_2,b_2}(t+C \eps \lambda^2, s) &\ge v_{a_2,b_2}(t,s) + C \eps \lambda^2 U_t(t,s) \\
&\ge U(1+\frac {a_2}{4}\rho + \frac {b_2} {2} t + C \eps \frac{\lambda^2}{\rho})\\
&\ge U(1+\frac {a_1}{4}\rho + \frac {b_1} {2} t )\\
 &\ge v_{a_1,b_1}(t,s),
 \end{align*}
and our claim is proved.

Since $v_{a_2,b_2}$ is increasing in the first coordinate, we obtain that
$$v_{a_1,b_1}(t,s) \le v_{a_2,b_2}(t+C \eps, s), \quad |(t,s)| \leq 1.$$
On the other hand, from the hypotheses on $h_i$ we see that in $B_1$ 
$$ t_1+ C \eps \le \bar t_2,$$ where $\bar t_2$ is the distance to $\mathcal S_2-C' \eps  e_n,$ for some $C'$ large depending on the $C$ above.
Hence in $B_1$ we have $$V_1(X)= v_{a_1,b_1}(t_1,s) \le v_{a_2,b_2}(t_1+C \eps, s) \le v_{a_2,b_2}(\bar t_2 , s) =V_2(X+C' \eps  e_n).$$
\end{proof}

\section{Harnack Inequality}

In this section we state and prove a Harnack type inequality for solutions to our free boundary problem \eqref{FB}. This will allow us to obtain some compactness of flat solutions after the transformation $g \rightarrow \tilde g$ (see Corollary \ref{corHI}) which is a crucial ingredient in Theorem \ref{mainT}.

\begin{thm}[Harnack inequality]\label{mainH} There exist $\bar \eps> 0$ small  and $\bar C>0$ large universal, such that if $g$ solves \eqref{FB} and it satisfies 
\begin{equation}\label{flatH}V(X +a_0e_n)  \leq g(X) \leq V(X+ b_0e_n) \quad \textrm{in $B_\rho(X^*) \subset B_1, $}\end{equation}with $V=V_{M, \xi', a,b} \in \mathcal V_\delta^0$, and 
$$\bar C \delta^2 \leq \frac{b_0 - a_0}{\rho} \leq \bar \eps, $$  with $|a_0|, |b_0| \leq 1,$ then \begin{equation}\label{impr}V(X + a_1e_n)  \leq g(X) \leq V(X+b_1e_n) \quad \textrm{in $B_{\bar \eta \rho}(X^*)$, }\end{equation} with $$a_0 \leq a_1\leq b_1 \leq b_0, \quad b_1-a_1 = (1-\bar \eta)(b_0-a_0),$$ for a small universal constant $\bar \eta \in (0,1/2)$. \end{thm}

In the particular case when $V=U$, this statement was proved in \cite{DR}. Our proof follows the same lines as the one in \cite{DR}  but it requires a more careful analysis since the function $V$ is no longer a precise solution.

From this statement we get the desired corollary to be
used in the proof of our main result. Precisely, assume $g$ satisfies
\eqref{flatH} in $B_1$ with $a_0=-\eps, b_0=\eps$ for some small $\eps 
\ll \bar \eps,$ and $\delta$ such that $\bar C \delta^2 \leq \eps.$ 
Notice that from Remark \ref{Visflat}, the functions $V$ and $g$ are $(4\delta+\eps)$-flat in $B_1.$

Then at any point $X^* \in B_{1/2}$ we can apply Harnack inequality
repeatedly for a sequence of radii $\rho_m=\frac 1 2 \bar \eta^m$ and obtain $$\label{osc2} V(X +a_m e_n) \leq g(X) \leq V(X+b_me_n) \quad \textrm{in $B_{\frac 1 2 \bar\eta^{m}}(X^*)$, }\ $$with
\begin{equation}\label{osctilde}b_m-a_m = (b_0-a_0)(1-\bar \eta)^m = 2 \eps (1-\bar \eta)^m, \end{equation} for all $m$'s, $m\geq 1$ such that 
\begin{equation}\label{m}4\eps \frac{(1-\bar \eta)^{m-1}} {\bar \eta^{m-1}} \leq 
\bar \eps.\end{equation}This implies that for all such $m$'s, the function $\tilde g$ satisfies
\begin{equation}\label{oscg}
\tilde V+a_m \leq \tilde g \leq \tilde V+b_m, \quad \textrm{in $B_{\frac 1 2 \bar \eta^{m} - 4\delta-\eps}(X^*) \setminus P,$}
\end{equation} with $a_m, b_m$ as in \eqref{osctilde}. 
Define the following (possibly multivalued) function 
\begin{equation}\label{newdeftildeg}
\tilde g_{\eps, V}(X):=  \frac{\tilde g(X)-\tilde V(X)}{\eps}, \quad X \in B_{1-4\delta-\eps} \setminus P,\end{equation} and notice that $$|\tilde g_{\eps, V}| \leq 1.$$

In view of \eqref{oscg} we then get that in $B_{\frac 1 4 \bar \eta^{m} }(X^*) \setminus P$
\begin{equation}\label{oscgev}\textrm{osc} \ \tilde g _{\eps, V} \leq 2(1-\bar \eta)^m,\end{equation}
provided that 
\begin{equation}\label{newbound}
4\delta+\eps \le  \eps^{1/2} \le \bar \eta^m/4.
\end{equation}
If $ \eps \le  \, \bar  \eps \, \, \bar \eta^{2 m_0}$ for some nonnegative integer $m_0$ then our inequalities above \eqref{m}, \eqref{newbound} and hence also \eqref{oscg} hold for all $m \le m_0$. We thus obtain the following corollary.

\begin{cor} \label{corHI}Let $g$ solve \eqref{FB} and satisfy for $\eps \leq \bar \eps$
$$V(X - \eps e_n) \leq g(X) \leq V(X+\eps e_n) \quad \textrm{in $B_1$,}$$ with $$V=V_{M, \xi', a,b} \in \mathcal V_\delta^0, \quad \bar C \delta^2 \leq \eps,$$ for $\bar \eps, \bar C>0$ universal constants.
If $$\eps \le \bar \eps \, \, \bar \eta^{2 m_0},$$ for some nonnegative integer $m_0$ (with $\bar \eta>0$ small universal), then the function $\tilde g_{\eps, V}$ defined in \eqref{newdeftildeg} satisfies $$a_\eps (X)\leq \tilde g_{\eps,V}(X) \leq b_\eps(X), \quad \textrm{in $B_{1/2} \setminus P$}$$  with
$$ b_\eps - a_\eps \leq 2(1 - \bar \eta)^{m_0},$$
and $a_\eps, b_\eps$ having a modulus of continuity bounded by the H\"older function $\alpha t^\beta$ for $\alpha, \beta$  depending only on $\bar \eta$. 
\end{cor}
The proof of Harnack inequality will follow from the Proposition below.

\begin{prop}\label{babyH} There exist $\bar \eps, \bar \delta > 0$  and $\bar C>0$ universal, such that if $g$ solves \eqref{FB} and  it satisfies 
\begin{equation}\label{flatH2} V(X  - \eps e_n)  \leq g(X) \leq V(X +\eps e_n) \quad \textrm{in $B_1, $} \quad \textrm{for $0< \eps \leq \bar \eps$}\end{equation}with $$V=V_{M,\xi', a,b} \in \mathcal V_\delta^0,\quad \delta \leq \bar \delta, \quad \bar C \delta^2 \leq \eps, $$
then either
$$g(X) \leq V(X+(1-\eta)\eps e_n)\quad \text{in $B_{\eta},$}$$ 
or 
$$g \geq V(X  -(1-\eta)\eps e_n) \quad \text{in $B_{\eta},$}$$
 for a small universal constant $\eta \in (0,1)$.
 \end{prop} 

 First we show that if $g \geq V$ and they separate of order $\eps$ at one point, then they separate also of order $\eps$ away from a neighborhood of $L=\{x_n=0, s=0\}$. This follows from the boundary Harnack inequality. Below are the details.
 
  \begin{lem}\label{prelimH} If $g$ solves \eqref{FB} and it satisfies 
\begin{equation}\label{flatH3}g(X) \geq V(X-\eps e_n) \quad \textrm{in $B_1, $}\end{equation} \begin{equation}\label{flatH4}g(\bar X) \geq V(\bar X) \quad \textrm{at some $\bar X \in B_{\frac 1 8}(\frac 1 4 e_n), $}\end{equation}with $V=V_{M,\xi', a,b} \in \mathcal V_\delta^0$, $\bar C \delta^2 \leq \eps$ for $\bar C>0$ universal,
then\begin{equation}\label{flatH5}g(X) \geq V(X-(1-\tau) \eps e_n) \quad \textrm{in $\mathcal C$,}\end{equation} with $$\mathcal C : =\{(x',x_n,s): \frac{d}{2} \leq |(x_n,s)| \leq \frac 1 2, |x'|\leq \frac 1 2\}, \quad d=\frac{1}{8\sqrt{n-1}}$$ and $\tau$ a small universal constant $\tau\in(0,1).$
\end{lem}

\begin{proof}
We have 
$$V(X-(1-\tau) \eps e_n) = V(X-\eps e_n) + \tau \eps V_n(X+\lambda \eps e_n), $$ for some $\lambda$ with $|\lambda| <1.$
Hence by \eqref{diadic},\eqref{concl1} for $\eps$ small enough and $X \in \mathcal C$
\begin{align*}
V(X-(1-\tau) \eps e_n) & \leq V(X-\eps e_n) + C\tau \eps U_n(X+\lambda \eps e_n)\\
 &  \leq  V(X-\eps e_n) + C_1\tau \eps U_n(X).
 \end{align*}

Thus,  if $h(X):= g(X)-V(X-\eps e_n)$  we need to show that 
\begin{equation}\label{finalh} h \geq c_1\eps  U_n, \quad \textrm{in $\mathcal C$,} \end{equation}
and then choose $\tau=c_1/C_1.$

To obtain \eqref{finalh}, notice that by a similar computation as the one above in view of 
\eqref{flatH4} and \eqref{diadic},\eqref{concl1} we get that for $\eps$ small enough
\begin{equation}\label{honepoint}h(\bar X)  \geq V(\bar X) - V(\bar X - \eps e_n )\geq c U_n(\bar X)\eps \geq c_2 \eps.
\end{equation}
Also, by \eqref{flatH3} we have
$$h \geq 0 \quad \textrm{in $B_1$.}$$
Finally, by \eqref{estun}
$$|\Delta h| \leq C \delta^2 U_n  \leq C_2 \delta^2, \quad \textrm{in $\widetilde{\mathcal C} \setminus P$},$$ where $\widetilde{\mathcal C} \supset\supset \mathcal C$ is the $d/4$-neighborhood of $\mathcal C.$

Thus in view of \eqref{honepoint} and Harnack inequality we have that (for $\bar C$ large enough)
\begin{equation}\label{firstbound}h \geq c_2 \eps - C \delta^2 \geq c_3 \eps, \quad \textrm{in $B_{1/8}(\frac 1 4 e_n)$.}\end{equation}
Denote by $$D:=\widetilde{\mathcal C} \setminus (B_{1/8}(\frac 1 4 e_n)\cup P)$$
and let $q_1,q_2$ satisfy in $D$
\begin{equation} \Delta q_1=0, \quad \Delta q_2= -1\end{equation} with boundary conditions respectively 
$$q_1=0 \quad \textrm{on $\p \widetilde{\mathcal C} \cup P$}, \quad q_1=1 \quad \textrm{on $\p B_{1/8}(\frac 1 4 e_n)$}$$ and
$$q_2=0 \quad \textrm{on $\p D.$}$$
By boundary Harnack inequality, $q_1$ is comparable to the distance function $s$ in a neighborhood of $P \cap \mathcal C \subset \subset \widetilde{\mathcal C}.$ Since $q_2$ is Lipschitz continuous in a neighborhood of $P \cap \mathcal C$, we then obtain 
\begin{equation}\label{q1q2}q_1 \geq c_4 q_2 \quad \textrm{in $\mathcal C \setminus B_{1/8}(\frac 1 4 e_n),$}\end{equation} with $c_4>0$ universal.
By the maximum principle,
$$h \geq q:= c_3 \eps \,q_1 - C_2\delta^2 \, q_2 \quad \textrm{in $D$},$$ since $h \geq q$ on $\p D$ and $\Delta h \leq \Delta q$ in $D.$ Hence, by \eqref{q1q2} we get that (for $\bar C$ large enough)
$$h \geq \eps \frac{c_3}{2}q_1 \geq c_5 \eps U_n, \quad \textrm{in $\mathcal C \setminus B_{1/8}(\frac 1 4 e_n),$}$$
where in the last inequality we used that (by boundary Harnack inequality) $q_1$ and $U_n$ are comparable. This inequality together with \eqref{firstbound} gives the desired claim \eqref{finalh}.

\end{proof}

We are now ready to present the proof of Proposition \ref{babyH}.

\

\textit{Proof of Proposition \ref{babyH}.} Assume that 
\begin{equation} g (\bar X) - V (\bar X)\geq 0, \quad \bar X= \frac 1 2 e_n.\end{equation}
Then in view of assumption \eqref{flatH2}  from Lemma \ref{prelimH}, after the change of variables $g \rightarrow \tilde g $ we get that 
\begin{equation}\label{firstineq}\tilde g(X) \geq\widetilde V(X)+\tau \eps  -\eps\quad \textrm{in $\mathcal C' \setminus P$}
\end{equation}
with$$\mathcal C' : =\{(x',x_n,s): d \leq |(x_n,s)| \leq \frac 1 4, |x'|\leq \frac 1 2\}, \quad d=\frac{1}{8\sqrt{n-1}}.$$ 
Denote by
$$W(X) : = V_{M+\frac{c}{n-1}\eps I, \xi', a,b+2c\eps}(X) \in  \mathcal{V}_{\delta+\eps},$$ with $c$ small to be made precise later.
Then in view of Proposition \ref{estimate} we have
\begin{equation}\label{VW}-2C_1(\delta+\eps)^2 \leq (\tilde V -\tilde W)+ c\eps(2 rx_n -\frac{1}{2(n-1)} |x'|^2) \leq 2C_1(\delta+\eps)^2.\end{equation}

First we choose $c$ small depending on $\tau$ such that $$\tilde V \geq \tilde W-\frac{\tau}{2} \eps,$$ where we used that $\bar C \delta^2 \leq \eps \leq \bar \eps$ with $\bar C \geq C(\tau)$ and $\bar \eps$ small enough. Then, if $\bar C$ is sufficiently large depending on $c$, 
\begin{equation}\label{lateral}\tilde V \geq \tilde W + \tau^* \eps, \quad \textrm{on $\{|(x_n,s)| \leq d, |x'| = \frac 1 2\} \setminus P,$}\end{equation} for some $\tau^*>0$ small, say $\tau^*<\tau/2$. 
These combined with \eqref{firstineq} give
\begin{equation}\label{secondineq}\tilde g(X) \geq\widetilde W(X)+\tau^*\eps   -\eps\quad \textrm{in $(\mathcal C' \cup \{|(x_n,s)| \leq d, |x'| = \frac 1 2\} )\setminus P$}.
\end{equation} Moreover, if $\bar C$ is large enough we get that
$W$ satisfies \eqref{Vsub} and hence $W$ is a subsolution. Thus from Lemma \ref{linearcomp} and the inequality above we conclude that
\begin{equation}\label{thirdineq}\tilde g(X) \geq\widetilde W(X)+\tau^*\eps   -\eps\quad \textrm{in $\{|(x_n,s)| \leq d, |x'| \leq \frac 1 2\} \setminus P$}.\end{equation}
Finally, from \eqref{VW} we see that there is a small neighborhood around the origin $B_\eta \subset \{|(x_n,s)| \leq d, |x'| \leq \frac 1 2\} $ ($\eta$ small universal depending on the constants above, $\eta <\tau^*/2$) such that  $$\tilde W \ge \tilde V - \frac{\tau^*}{2} \eps, \quad \textrm{in $B_{2\eta} \setminus P.$}$$

Hence, from \eqref{thirdineq} we conclude that $$\tilde g \ge \tilde V +\eta \eps -\eps \quad \mbox{in} \quad B_{2\eta} \setminus P,$$
for some small universal constant $\eta$, and the lemma is proved after the change of variable $\tilde g \rightarrow g.$
\qed

\

We conclude this section with the proof of Theorem \ref{mainH}.

\

\textit{Proof of Theorem \ref{mainH}.} After a translation of the origin we may assume that we satisfy our flatness hypothesis \eqref{flatH} in $B_\rho(X^*) \subset B_2$ with 
$$(x^*)'=0, \quad a_0+b_0=0, \quad V \in \mathcal V_{2\delta}^0.$$
We dilate the picture by a factor of $2/\rho$ and work with the rescalings
$$g_\rho(X) =(\frac \rho 2)^{-1/2} g(\frac{\rho}{2} X), \quad V_\rho(X) = (\frac \rho 2)^{-1/2} V(\frac{\rho}{2} X), $$ which are defined in a ball of radius 2 included in $B_{4/\rho}.$
 Notice that, if $V \in \mathcal V_{2\delta}^0$ then $V_\rho \in \mathcal V_{2 \delta}^0.$ 
 
 After dropping the subindex $\rho$ for simplicity of notation, we may assume that the flatness condition \eqref{flatH} holds in some ball $B_2(X^*) \subset \R^{n+1}$, 
with $V \in \mathcal V_{2 \delta}^0,$ $$a_0= -\eps, \quad b_0= \eps, \quad  (x^*)'=0$$ and 
$$\bar C \delta^2 \leq \eps \leq \bar \eps.$$
We need to prove the conclusion \eqref{impr} in a ball $B_{2\bar \eta}(X^*).$

We distinguish three cases depending on whether $X^*$ is close to $L$, close to $P$, or far from $P$. 

In Case 2 and Case 3 we will use the following properties from Remark \ref{remark}.

\begin{equation}\label{newremark}c \leq \frac{V_n}{U_n} \leq C, \quad |\triangle V| \le C \delta^2 U_n \quad \mbox{in} \quad B_2(X^*)\setminus (P \cup \{|(x_n,s)|\le 20 \delta |x'|\}). \end{equation}

Below $\eta$ is the universal constant from Proposition \ref{babyH}. 

\

\textit{Case 1.} $|X^*|  < \eta/4.$ 

In this case, since $B_1 \subset B_2(X^*)$ we follow under the assumptions of Proposition \ref{babyH}. Hence we can conclude that for any $\bar \eta \leq \eta/4$
in  $B_{2\bar \eta}(X^*) \subset B_{\eta}$ either 
$$g(X) \leq V(X +(1-\eta)\eps e_n),$$ 
or 
$$g(X) \geq V(X  -(1-\eta)\eps e_n),$$ and our conclusion is satisfied for all $\bar \eta \leq \eta/4.$ 

\

\textit{Case 2.} $|X^*|  \geq \eta/4, $ and $B_{\frac{\eta}{32}}(X^*) \cap P = \emptyset.$ 

In this case, if $\bar \eps $ is small enough then it follows from \eqref{newremark} that  the function $$h(X):= g(X) - V(X-\eps e_n) \geq 0,$$ satisfies  $$|\Delta h| \leq C \delta^2 U_n \quad \textrm{in $B:=B_{\frac{\eta}{64}}(X^*)$}.$$  Notice also that by  Harnack inequality
\begin{equation}\label{claim2new}\frac{U_n(X)}{U_n(Y)} \leq  C \quad\textrm{for $X,Y \in B$,} \end{equation} with $C$ universal.
Assume that 
$$g(X^*) \geq V(X^*).$$
Then, in view of \eqref{newremark} and \eqref{claim2new} 
$$h(X^*) = g(X^*) - V(X^*-\eps e_n) \geq  c \eps U_n(X^*).$$

Hence by Harnack inequality, \eqref{claim2new} and the condition $\bar C\delta^2 \leq \eps$

$$h \geq c \eps U_n(X^*) -  C \delta^2\|U_n\|_{L^\infty(B)} \geq c' \eps  U_n(X^*) \quad \textrm{in $B_{\frac{\eta}{128}}(X^*)$}.$$

Thus, using \eqref{newremark} we have that for $\tau$ small enough
$$h \geq c' \eps  \sup_B V_n \geq V(X- (1-\tau)\eps e_n) -V(X-\eps e_n)\quad \textrm{in $B_{\frac{\eta}{128}}(X^*)$},$$ from which our desired conclusion follows with any $\bar \eta$ such that
$2\bar \eta \leq \min\{\eta/128,\tau\}.$

\

\textit{Case 3.} $|X^*| \geq \eta/4$ and $B_{\frac{\eta}{32}(X^*)} \cap P \neq \emptyset.$

In this case we argue similarly as in the previous case but we need to make use of the boundary Harnack inequality. 

Assume that $X^* \in \{s>0\}$ and call $X^*_0=(x^*, 0)$ the projection of $X^*$ onto $\{s=0\}$. If $\bar \eps$ is small enough then it follows from  \eqref{newremark} that  the function $$h(X):= g(X) - V(X-\eps e_n) \geq 0,$$ satisfies  $$|\Delta h| \leq C \delta^2 U_n \quad \textrm{in $B:=B_{\frac{\eta}{8}}(X^*_0) \cap \{s>0\}$},$$ for a universal constant $C.$ Denote by $Y^*= X^*_0 + \frac{\eta}{16} e_n$ and assume that 
$$g(Y^*) \geq V(Y^*).$$ As in the previous case, by Harnack inequality
\begin{equation}\label{maxbound}h \geq c \eps U_n(Y^*) \quad \textrm{in $B_{\frac{\eta}{32}}(Y^*).$}\end{equation}

Now we argue similarly as in Lemma \ref{prelimH}.

Denote by $$D:= (B_{\eta/8}(X^*_0) \setminus B_{\eta/32}(Y^*)) \cap \{s>0\}.$$ Let $q_1,q_2$ satisfy in $D$
$$\Delta q_1=0, \quad \Delta q_2=-1$$
with boundary conditions respectively,
$$q_1=1 \quad \textrm{on $\p B_{\eta/32}(Y^*)$}, \quad q_1=0 \quad \textrm{on $\p (B_{\eta/8}(X^*_0)\cap \{s>0\}) $}$$
and
$$q_2=0 \quad \textrm{on $\p D.$}$$

By the maximum principle, in view of \eqref{maxbound} we obtain that
$$h \geq c\eps U_n(Y^*)q_1-C\delta^2 q_2 \quad \textrm{in $D$.}$$ Moreover, $$q_1 \geq c q_2 \quad \textrm{in $D \cap B_{\eta/16}(X_0^*).$}$$ Hence using that $\bar C\delta^2 \leq \eps$ we get
$$h(X) \geq c'\eps U_n(Y^*) q_1(X) \geq c \eps U_n(X) \quad \textrm{in $B_{\eta/16}(X^*_0) \cap \{s>0\}$}$$
where in the last inequality we used that $U_n(Y^*)q_1$ is comparable to $U_n$ in view of boundary Harnack inequality.

Now we use \eqref{diadic} and \eqref{newremark} to conclude 
\begin{align*}
h(X)&= h(x,x_{n+1}) \ge c \eps \sup_ {B_{\frac \eta 8}(X_0^*)}U_n(y,x_{n+1}) \ge c \eps \sup_{B_{\frac \eta 8}(X_0^*)}V_n(y,x_{n+1})  \\
& \geq V(X- (1-\tau)\eps e_n) -V(X-\eps e_n)\quad \textrm{in $B_{\frac{\eta}{16}}(X_0^*)\supset B_{\frac{\eta}{32}}(X^*)$}.\end{align*} Then our desired statement holds for $\bar \eta \leq \min\{\tau/2, \eta/64\}.$
\qed

\section{Improvement of flatness.}

In this section we prove our main Theorem \ref{mainT}. We start with the following quadratic improvement of flatness proposition. We show that if a solution $g$ stays in a $\lambda^{2+\alpha}$ neighborhood of a function $V \in \mathcal V_1^0$ in a ball $B_\lambda$ then in $B_{\eta \lambda},$ $g$ is in a $(\lambda\eta)^{2+\alpha}$ neighborhood of another function $V$ in the same class. 

\begin{prop}\label{impflat}Given $\alpha \in (0,1)$, there exist $\lambda_0, \eta_0 \in (0,1)$ and $C>0$ large depending on $\alpha$ and $n$, such that if $g$ solves \eqref{FB}, $0 \in F(g)$ and $g$ satisfies
\begin{equation}\label{a1} V(X - \lambda^{2+\alpha} e_n) \leq g(X) \leq V(X + \lambda^{2+\alpha} e_n), \quad \textrm{in $B_{\lambda}$ with $0<\lambda\leq \lambda_0$}
\end{equation} for $V= V_{M, 0, a, b} \in \mathcal{V}^0_1$, then in a possibly different system of coordinates denoted by $\bar E = \{\bar e_1,\ldots, \bar e_n, \bar e_{n+1}\}$, 
\begin{equation}\label{c1} \bar V(X - (\eta_0 \lambda)^{2+\alpha} \bar e_n) \leq g( X) \leq \bar V(X + (\eta_0\lambda)^{2+\alpha} \bar e_n), \quad \textrm{in $B_{\eta_0\lambda}$}
\end{equation} for some $\bar V= V_{\bar{\mathcal S}, \bar a, \bar b}$ (defined in \eqref{vS}) with $\bar S$ given in the $\bar E$ coordinates by 
$$\bar{\mathcal S}= \{\bar x_n = \frac 1 2 (\bar x')^T \bar M \bar x'\},$$ and $$\|\bar M - M\|, |\bar a -a|, |\bar b -b| \leq C\lambda^\alpha, \quad \bar a + \bar b -tr \bar M=0.$$
Moreover, for any $\sigma \in (0,1]$, the surfaces $\bar {\mathcal S}$ and $\mathcal S$ separate in $B_\sigma$ at most 
$C (\lambda^\alpha \sigma^2 +\lambda^{1+\alpha} \sigma).$ 
\end{prop} 
\begin{proof} Let $\eta_0, C$ be the constants in Corollary \ref{classnewcor}.

The proof is by compactness. Assume that no such $\lambda_0$ exists, then we can find a sequence 
of $\lambda_k$'s, tending to 0, $g_k$ and $V_k$ satisfying \eqref{a1} for which \eqref{c1} fails. We rescale $g_k$ and $V_k$. For simplicity of notation we drop the dependence on $k$ and denote
$$g_\lambda(X) = \lambda^{-1/2} g(\lambda X), \quad V_\lambda(X)=\lambda^{-1/2}V(\lambda X), \quad X\in B_1.$$ Notice that $$V_\lambda= V_{\lambda M,0, \lambda a, \lambda b} \in \mathcal V_\lambda^0,$$ and 

\begin{equation*} V_\lambda(X - \lambda^{1+\alpha} e_n) \leq g_\lambda(X) \leq V_\lambda(X + \lambda^{1+\alpha} e_n)\quad \textrm{in $B_{1}.$}
\end{equation*} 
Let  $$\eps =\lambda^{1+\alpha}, \quad \delta=\lambda$$ and define 
\begin{equation}\label{wlambda}w_\lambda:= \frac{\tilde g_\lambda - \gamma_{V_\lambda}}{\eps}.\end{equation} Thus by Proposition \ref{estimate}
$$w_\lambda=  \frac{\tilde g_\lambda -\tilde V_\lambda}{\eps} + \frac{\tilde V_\lambda - \gamma_{V_\lambda}}{\eps}= \widetilde{(g_\lambda)}_{\eps, \tilde V_\lambda} + O(\frac{\delta^2}{\eps})$$
and hence by Corollary \ref{corHI} we get that $w_\lambda$ converges uniformly to a Holder continuous function $w_0$ as $k$ tends to $\infty$ (and $\lambda \rightarrow 0$), with $w_0(0)=0$ and $|w_0| \leq 1.$ 

We claim that $w_0$ is a viscosity solution of the linearized problem 

\begin{equation}\label{linearp}\begin{cases} \Delta (U_n w_0) = 0, \quad \text{in $B_{1/2} \setminus P,$}\\ |\nabla_r w_0|=0, \quad \text{on $B_{1/2}\cap L$.}\end{cases}\end{equation}

We start by showing that $U_n w_0$ is harmonic in $B_{1/2} \setminus P. $ 

Let $\tilde \varphi$ be a smooth function which touches $w_0$ strictly from below at $X_0 \in B_{1/2} \setminus P.$ We need to show that 
\begin{equation}\label{des} \Delta (U_n \tilde \varphi)(X_0) \leq 0.
\end{equation}
Since  $w_\lambda$ converges uniformly to $w_0$ in $B_{1/2} \setminus P$ we conclude that there exist a sequence of constants $c_\lambda \rightarrow 0$ and a sequence of points $X_\lambda \in B_{1/2} \setminus P$, $X_\lambda \rightarrow X_0$ such that $\tilde \psi_\lambda:=\eps(\tilde \varphi + c_\lambda) + \gamma_{\tilde V_\lambda}$ touches $\tilde g_\lambda$ by below at $X_\lambda$ for a sequence of $\lambda$'s tending to 0.

Define the function $\psi_\lambda$ by the following identity
\begin{equation}\label{varphi}\psi_\lambda (X-  \tilde{\psi_\lambda}(X) e_n) = U(X). \end{equation}

Then according to \eqref{gtildeg} $\psi_\lambda$ touches $g_\lambda$ from below at $Y_\lambda = X_\lambda -  \tilde \psi_\lambda(X_\lambda)e_n \in B_1^+(g_\lambda)$. Thus, since $g_\lambda$ satisfies \eqref{FB} in $B_1$ it follows that
\begin{equation}\label{sign}
\Delta \psi_\lambda(Y_\lambda) \leq 0.
\end{equation}
 In a neighborhood of $X_0$,  $\gamma_{V_\lambda}/\lambda$ has bounded $C^k$ norms (depending on $|X_0|$) hence $\tilde \psi_\lambda/\lambda$ has also bounded $C^k$ norms. By Proposition \ref{1}
\begin{align*}\Delta \psi_\lambda&=\lambda \Delta (U_n (\tilde \psi_\lambda/\lambda)) + O(\lambda^2)\\ &=\Delta (U_n \tilde \psi_\lambda)+O(\lambda^2)\\
& = \Delta (U_n (\eps \tilde \varphi + \gamma_{\tilde V_\lambda}))(X_\lambda)+ O(\lambda^2)\\ & = \eps \Delta(U_n \tilde \varphi) + O(\lambda^2)\end{align*}
where we have used that $$\Delta(U_n\gamma_{V_{\lambda}}) = 0.$$
This can be checked either explicitly or by using Theorem \ref{class}.

In conclusion
$$\eps \Delta (U_n \tilde \varphi)(Y_\lambda) + O(\lambda^2) \le 0.$$
We divide by $\eps=\lambda^{1+\alpha}$ and let $\lambda \rightarrow 0$. Using that $Y_\lambda \rightarrow X_0$ we  obtain
$$\Delta(U_n \tilde \varphi)(X_0) \le 0,$$ as desired.

Next we need to show that $$|\nabla_r w_0 |(X_0)= 0, \quad X_0=(x'_0,0,0) \in B_{1/2} \cap L,$$ in the viscosity sense of Definition \ref{linearsol}. 

We argue by contradiction.  Assume for simplicity (after a translation) that there exists a function $\phi$ which  touches $w_0$ by below at $0$ with  $\phi(0)=0$ and such that $$\phi(X) = \xi' \cdot x'  + \beta  r + O(|x'|^2 + r^{3/2}), $$ with $$\beta >0.$$   

Then we can find constants $\sigma, \tilde r$ small and $A$ large  such that the polynomial

$$q(X)= \xi' \cdot x'  - \frac{A}{2}|x'|^2 + 2A (n-1)x_n r$$
 touches $\phi$ by below at $0$  in a tubular neighborhood  $N_{\bar r}= \{|x'|\leq \tilde r, r \leq \tilde r\}$ of $0,$ with 
 
 $$\phi- q \geq \sigma>0, \quad \text{on $N_{\tilde r} \setminus N_{\tilde r/2}$.}$$ This implies that
\begin{equation}\label{second}
w_0 - q \geq \sigma>0, \quad \text{on $N_{\tilde r} \setminus N_{\tilde r/2}$,}\end{equation}
and 
\begin{equation}\label{third}
w_0(0)- q(0) =0.\end{equation}
In particular, by continuity near the origin we can find a point $X^*$ such that \begin{equation}\label{thirdprime}
w_0 (X^*)- q(X^*) \leq \frac{\sigma}{8}, \quad \textrm{$X^* \in N_{\tilde r} \setminus P$ close to 0}. \end{equation}

Now, let us define 

$$W_\lambda:=V_{\lambda M+A\eps I, -\eps \xi', \lambda a, \lambda b+2\eps A(n-1)} \in \mathcal V_{2\delta}.$$

Then in view of Proposition \ref{estimate} we have 

$$\widetilde W_\lambda = \eps q + \widetilde \gamma_{V_\lambda} + O(\delta^2)$$
and moreover, $W_\lambda$ is a subsolution to our problem since $\eps \gg \delta^2$.

Thus, from the uniform convergence of $w_\lambda$ to $w_0$ and \eqref{second} we get that (for all $\lambda$ small)
\begin{equation}\label{fcont}
 \frac{\tilde g_\lambda - \widetilde W_\lambda}{\eps} = w_\lambda - q + O(\frac{\delta^2}{\eps}) \geq \frac \sigma 2 \quad \text{in $(N_{\tilde r} \setminus N_{\tilde r/2}) \setminus P.$} 
\end{equation}
Similarly, from the uniform convergence of $w_\lambda$ to $w_0$ and \eqref{thirdprime} we get that for $k$ large
\begin{equation}\label{scont}
\frac{(\tilde g_\lambda - \widetilde W_\lambda)(X^*)}{\eps} \leq \frac \sigma 4,  \quad \text{ at  $X^* \in N_{\tilde r} \setminus P.$}\end{equation} 

On the other hand, it follows from Lemma \ref{linearcomp} and \eqref{fcont} that
\begin{equation*}
\frac{\tilde g_\lambda - \widetilde W_\lambda}{\eps} \geq \frac \sigma 2 \quad \text{in $N_{\tilde r} \setminus P,$} 
\end{equation*} which contradicts \eqref{scont}.

In conclusion $w_0$ solves the linearized problem. Hence, by Corollary \ref{classnewcor} since $w_0(0) = 0$, $w_0$ satisfies 
\begin{equation}\label{bound2}  - \frac{1}{4} \eta_0^{2+\alpha}\leq  w_0(X) -(  \xi_0 \cdot x' + \frac 12 x'^TM_0x' -  \frac{a_0}{2} r^2 - b_0 x_n r) \leq  \frac 14 \eta_0^{2+\alpha} \quad \text{in $B_{4\eta_0},$}\end{equation} 
 for some $\eta_0 \in (0,1)$ universal and with $$a_0+b_0 - tr M_0=0, \quad \quad |\xi_0|,\|M_0\|, |a_0|, |b_0| \le C.$$

From the uniform convergence of $w_\lambda$ to $w_0$, we get that for all $k$ large enough
\begin{equation}\label{boundnew}  - \frac{1}{2} \eta_0^{2+\alpha}\leq  w_\lambda(X)-\frac{\tilde T_\lambda-\gamma_{\tilde V_\lambda}}{\eps} \leq \frac{1}{2} \eta_0^{2+\alpha} \quad \text{in $B_{4\eta_0} \setminus P,$}\end{equation}  
with
$$T_\lambda:=V_{\lambda M-\eps M_0, -\eps \xi_0,\lambda a-\eps a_0, \lambda b-\eps b_0}.$$
In conclusion, from the definition \eqref{wlambda} of $w_\lambda$, we get
\begin{equation}\label{boundnew2}  
\tilde T_\lambda -\frac{\eps}{2} \eta_0^{2+\alpha}  \le \tilde g_\lambda \le \tilde T_\lambda +\frac{\eps}{2} \eta_0^{2+\alpha},
\end{equation} 
or
$$T_\lambda(X-\frac{\eps}{2} \eta_0^{2+\alpha}e_n)\le  g_\lambda(X) \le T_\lambda(X+\frac{\eps}{2} \eta_0^{2+\alpha}e_n) \quad \mbox{in $B_{2\eta_0}$.}$$

We rescale $g_\lambda$ back from the ball $B_1$ to $B_\lambda$ and obtain
\begin{equation}\label{rotate}T(X-  \frac{\eps \lambda}{2} \eta_0^{2+\alpha}e_n)\le  g(X) \le T(X+\frac{\eps \lambda }{2} \eta_0^{2+\alpha}e_n) \quad \mbox{in $B_{2\lambda\eta_0}$,}\end{equation}
with $$T=V_{\mathcal S_T, a_T,b_T},$$
for
$$\mathcal S_T:= \{x_n= \frac 1 2 (x')^TM_T x' + \xi_T \cdot x'\},$$ $$M_T:=M-\frac \eps \lambda M_0, \quad \xi_T:= -\eps \xi_0, \quad a_T:= a- \frac \eps \lambda a_0, \quad b_T: = b- \frac \eps \lambda b_0 .$$

Next we show that in a different system of coordinates, called $\bar E$, the function $T$ can be approximated by $V_{M_T,0, a_T,  b_T}.$

Assume for simplicity that $\xi_T$ points in the $e_1$ direction. Then we choose an orthogonal system of coordinates  $\bar E:= \{\bar e _1, \bar e_2, \ldots, \bar e_{n+1}\}$ with $$\bar e_i=e_i, \quad \textrm{if $i \neq 1,n$}$$ and $\bar e_n$  normal to $S_T$ at 0.

Notice that the $\bar E$ system of coordinates is obtained from the standard one after an orthogonal transformation of norm bounded by $C|\xi_T|$ which is smaller than $C\eps.$

A point in this system on coordinates is denoted by $\bar X$. We let,
$$\bar {\mathcal{S}}: = \{\bar x_{n} = \frac 1 2 (\bar x ')^T M _T \bar x'\},$$ 
and we write $\bar{\mathcal S}$ as a graph in the $e_n$ direction, that is $$\bar{\mathcal S}:= \{ x_{n} = h(x')\}.$$ We claim that in a ball of radius $\sigma$  the distance (in the $e_n$ direction) between $\mathcal S_T$ and $\bar{\mathcal S}$ in $B_\sigma$ is less that $C \eps \sigma^2$, for any $0< \sigma \leq 1 $.

Indeed, since $\bar x = O x$ with $O$ orthogonal and $\|O- I\| \leq C\eps$, we obtain by implicit differentiation
$$\|D_{x'}^2 h - M_T\|_{L^\infty(B_1)} \leq C\eps, \quad \nabla_{x'} h(0)=\xi_T.$$ 
Thus
in $B_{2 \eta_0 \lambda}$ we have that the surfaces $$S_T \pm \frac \eps 2 \lambda \eta_0^{2+\alpha} e_n$$ lie between $$\bar S \pm \eps \lambda \eta_0^{2+\alpha} \bar e_n$$ since $C\eps (\eta_0 \lambda)^2 \ll \frac \eps 2 \lambda \eta_0^{2+\alpha}$.

In view of this inclusion, using that $v_{a_T,b_T}(t,s)$ is monotone in $t$, we obtain from \eqref{rotate} the desired conclusion  \eqref{c1} with $\bar M= M_T, \bar a = a_T, \bar b = b_T.$

Since the distance between $\mathcal S _T$ and $S$ in $B_\sigma$ is less than $C(\dfrac \eps \lambda \sigma^2+\eps \sigma)$ the proof is finished. 
\end{proof}

We can now prove our main Theorem \ref{mainT}. In fact we show that under our flatness assumption, a solution $g$ can be approximated in a $C^{2,\alpha}$ fashion by a function $V \in \mathcal V_C^0.$

\begin{thm}
There exists $\bar \eps>0$ small universal  such that if $g$ solves \eqref{FB} in $B_1$ with
\begin{equation}\label{ass1}\{x \in \mathcal{B}_1 : x_n \leq - \bar \eps\} \subset \{x \in \mathcal{B}_1 : g(x,0)=0\} \subset \{x \in \mathcal{B}_1 : x_n \leq \bar \eps \},\end{equation}
then in an appropriate system of coordinates denoted by $\bar e_i$
\begin{equation*}\label{trap1}V(\bar X-C\lambda^{2+\alpha}\bar e_n) \le g(\bar X) \le V(\bar X+C\lambda^{2+\alpha}\bar e_n) \quad \mbox{in $B_\lambda$, \, for all $0<\lambda<1/C,$} \end{equation*}
for some $V=V_{M_0,0,a_0,b_0} \in \mathcal V^0_C$, with $C$ depending on $n$ and $\alpha$.  In particular,  $F(g) \cap \mathcal{B}_{1/2}$ is a $C^{2,\alpha}$ graph in the $e_n$ direction for any $\alpha \in (0,1)$.
\end{thm}
\begin{proof}
It suffices to prove the theorem for any fixed $\alpha \in (0,1)$ for some $\bar \eps(\alpha)$, $C(\alpha)$ depending on $\alpha$. The dependence of $\bar \eps$ on $\alpha$ can be easily removed by fixing $\bar \eps:=\bar \eps(\bar \alpha)$, say with $\bar \alpha =1/2$. Then by the conclusion \eqref{trap1} for $\bar \alpha$, appropriate rescalings of $g$ satisfy the flatness assumption \eqref{ass1} also for $\bar \eps(\alpha)$ for any $\alpha \in (0,1)$. 

By Lemma \ref{differentassumption} the rescaling $$g_\mu(X)=\mu^{-\frac 12}g(\mu X)$$ satisfies 
$$U(X- \tau^{2+\alpha}e_n) \le g_\mu (X) \le U(X+\tau^{2+\alpha}e_n) \quad \mbox{in $B_{1}$,} $$
provided that $\bar \eps$, $\mu$ are chosen small depending on $\tau \le \lambda_0$, with $\lambda_0$ the universal constant in Proposition \ref{impflat} and $\tau$ small universal to be made precise later.
Thus $g_\mu$ satisfies in $B_{\tau}$ the hypotheses of Proposition \ref{impflat} with $M=0$, $a=0$, $b=0$. 
Then we can apply Proposition \ref{impflat} repeatedly  for all $\tau_k:=\tau \eta_0^k$ 
 since by choosing $\tau$ small enough we can guarantee that $\sum_k C \tau_k^\alpha \leq 1$ and hence the corresponding $M_k$, $a_k$, $b_k$ have always norm less than $1$. Thus we obtain
 \begin{equation}\label{iterate}V_{\mathcal S_k,a_k,b_k}(X- \tau_k^{2+\alpha}e^k_n) \le g_\mu (X) \le V_{\mathcal S_k,a_k,b_k}(X+\tau_k^{2+\alpha}e^k_n) \quad \mbox{in $B_{\tau_k}$.} \end{equation}

Using that $\mathcal S_k$ and $\mathcal S_{k+1}$ separate (in the $e_n$-direction) in $B_\sigma$ at most $C(\tau_k^\alpha \sigma^2 + \tau_k^{1+\alpha} \sigma)$
we conclude that as $k \to \infty$, the paraboloids $\mathcal S _k$ converge uniformly in $B_1$ to a limit parabolid $\mathcal S_*$. Moreover, $\mathcal S_*$ also separates from $S_k$ in $B_\sigma$ by at most  $C(\tau_k^\alpha \sigma^2 + \tau_k^{1+\alpha} \sigma)$ in the $e_n^*$ direction where $e_n^*$ is the normal to $S_*$ at the origin.
Finally, as $k \rightarrow \infty$, $a_k \to a_*$, $b_k \to b_*$, with $$|a_k - a_*|, |b_k-b_*| \le C \tau_k^{\alpha} .$$ 

Now notice that in $B_{2\tau_k}$, the paraboloids $S_k$ and $S_*$ separate at most $C \tau_k^{2+\alpha}$, thus we can apply Lemma \ref{distance} and use the inequality \eqref{iterate} to obtain
 $$V_{\mathcal S_*,a_*,b_*}(X-C\tau_k^{2+\alpha}e^*_n) \le g_\mu (X) \le V_{\mathcal S_*,a_*,b_*}(X+C\tau_k^{2+\alpha}e^*_n) , \quad \textrm{in $B_{\tau_k}$.}$$ Rescaling back we obtain the desired claim.
 
\end{proof}

\section{The regularity of the linearized problem}

We recall that the linearized problem associated to \eqref{FB} is 
\begin{equation}\label{linearnew}\begin{cases} \Delta (U_n h) = 0, \quad \text{in $B_1 \setminus P,$}\\ |\nabla_r h|=0, \quad \text{on $B_1\cap L$,}\end{cases}\end{equation}
where 
$$|\nabla_r h |(X_0) := \di\lim_{(x_n,s)\rightarrow (0,0)} \frac{h(x'_0,x_n, s) -h(x'_0,0,0)}{r}, \quad   r^2=x_n^2+s^2 .$$

In this section we obtain a second order expansion near the origin for a solution $h$ to \eqref{linearnew}.

\begin{thm}\label{class} Let  $h$ be a solution to \eqref{linearnew} such that $|h|  \leq 1$. Then $h$ satisfies
\begin{equation}\label{mainh}|h(X) - (h(0) +  \xi_0 \cdot x' + \frac 1 2 (x')^TM_0x'  -\frac{a_0}{2}r^2 -b_0rx_n)| \leq C |X|^{3},\end{equation} for some $a_0,b_0, \xi_0, M_0$ with $|\xi_0|, |a_0|,  |b_0|, \|M_0\| \leq C$, $C$ universal  and $$ a_0 +b_0 -tr M_0=0.$$ \end{thm}

\begin{proof}
This proof is a refinement of Theorem 8.1 in \cite{DR} where the authors obtained a first order expansion for $h$, in particular 
\begin{equation}\label{DRexp}|h(X) - h(X_0)|=O (|X-X_0|), \quad X_0  \in L.\end{equation} Also in \cite{DR} it is shown that $h$ and its derivatives of all orders in the $x'$ direction are Holder continuous  with norm controlled  by a universal constant in $B_{1/2}$ (see Corollary 8.7.)

We wish to prove that 
\begin{equation}\label{exp2} |h(x',x_n,x) - h(x', 0,0) +  \frac{a(x')}{2}r^2 + b(x')x_n r| \leq C r^{3}, \quad (x',0,0) \in B_{1/2} \cap L, \end{equation} with $C$ universal and $h(\cdot, 0,0), a,b$ smooth functions of $x'$.

The function $h$ solves $$\Delta(U_n h) =0 \quad \text{in $B_1 \setminus P$}, $$ and since $U_n$ is independent on $x'$ we can rewrite this equation as
\begin{equation}\label{2dreduction} \Delta_{x_n,s} (U_n  h) = - U_n \Delta_{x'} h.\end{equation}
Moreover, since $\Delta_{x'}h$ solves the same linear problem as $h$ then any estimate for $h$ also holds for $\Delta_{x'} h.$

For each fixed $x'$, we investigate the 2-dimensional problem 
$$\Delta (U_t h) = U_t f, \quad \text{in $B_{1/2} \setminus \{t \leq 0, s=0\} \subset \R^2$}$$
with $h,f \in C^{0,\beta}$. Without loss of generality, for a fixed $x'$ we may assume $h(x', 0,0)=0.$
Thus in view of \eqref{DRexp}, the function $$H:=U_t h$$ is continuous in $B_{1/2} \subset \R^2$ and satisfies
$$\Delta H = U_t f \quad \text{in $B_{1/2} \setminus \{t \leq 0, s=0\}$, }\quad H=0 \quad \text{on $B_{1/2} \cap \{t \leq 0, s=0\}.$}$$

Now, we consider the holomorphic transformation $z \rightarrow \frac 1 2 z^2$
$$\Phi: (\zeta,y) \rightarrow (t,s) =(\frac 1 2 (\zeta^2-y^2), \zeta y)$$
which maps  $B_{1} \cap \{\zeta > 0\}$ into $B_{1/2} \setminus  \{t \leq 0, s=0\}$ and call

$$\tilde h(\zeta,y) = h(t,s), \quad \tilde f(\zeta,y) = f(t,s), \quad \tilde H(\zeta,y) = H(t,s)$$ with $(\tilde r, \tilde \theta),$ the polar coordinates in the $(\zeta,y)$ plane.
Then, easy computations show that
\begin{equation}\label{reflect}\Delta \tilde H = \zeta \tilde f \quad \text{in  $B_{1} \cap \{\zeta > 0\}$}, \quad \tilde H(\zeta,y)=\frac{\zeta}{\tilde r^2}\tilde h,\end{equation} and $$\tilde H = 0 \quad \textrm{on $\{\zeta=0\}$.}$$
Since the right-hand side is in $C^{0,\beta}$ and $\tilde h, \tilde f$ have the same regularity, we conclude from repeatedly applying Lemma \ref{odd} below that $\tilde h,\tilde f \in C^{\infty}$ with $$\|\tilde f\|_{C^{k, \beta}(B_{1/2}^+)}, \|\tilde h\|_{C^{k, \beta}(B_{1/2}^+)} \leq C(k, \beta).$$ 
Notice that we can reflect $\tilde H$ oddly and $\tilde h, \tilde f$ evenly across $\{\zeta=0\}$ and the resulting functions will still solve \eqref{reflect} in $B_1$. Moreover from our assumptions, $\tilde f$ and $\tilde h$ are even with respect to $y$. Thus, we conclude that the Taylor polynomials for $\tilde f, \tilde h$ around the origin, are polynomials in $\zeta^2, y^2.$  
Now we use the Taylor expansion for $\tilde H$ around 0, which is odd with respect to $\zeta$ and even in $y,$ that is  
$$\tilde H (\zeta,y) = \zeta(d_0+d_1\zeta^2+d_2y^2+O(\tilde{r}^4))$$ 
with $$6d_1+2d_2=\tilde f(0)=f(0).$$Thus,
$$ \tilde h (\zeta,y) = \tilde \zeta^2(d_0+d_1\zeta^2+d_2y^2+O(\tilde r^4)).$$ In terms of the $(t,s)$ coordinates this means that
$$h(t,s)= 2r(d_0+2d_1 r (\cos\frac{\theta}{2})^2  + 2d_2 r (\sin\frac{\theta}{2})^2) + O(r^3)=$$
$$=  2r(d_0+(d_1+d_2) r  + (d_1 - d_2)t ) + O(r^3)$$
$$=  2d_0r-\frac a 2 r^2  - btr  + O(r^3).$$

In conclusion, 

$$|h(X) - h(x',0,0) - 2d_0(x')r+\frac{a (x')}{2} r^2  + b(x')x_nr| \leq C r^3 \quad \text{in $B_{1/2} $}$$ with $C$ universal, $$\|a\|_{L^\infty(\{|x'|\leq 1/2\})}, \|b\|_{L^\infty(\{|x'|\leq 1/2\})} \leq C,$$ and $$a+b=\Delta_{x' }h(x',0,0).$$ Since $h$ solves \eqref{linearnew} we must also have $d_0(x')=0$ and hence
$$|h(X) - h(x',0,0) + \frac{a(x')}{2} r^2  + b(x')x_nr| \leq C r^3, \quad \text{in $B_{1/2} $}.$$ 

Notice that $a, b$ are smooth functions of $x'$ with all order derivatives bounded by appropriate universal constants. Indeed due to the linearity of the problem it is easy to see that $D^\beta_{x'} a, D^\beta_{x'}b $ are the corresponding $a$ and $b$ for $D^\beta_{x'}h$.
Writing the Taylor expansions at 0 for $h(x', 0, 0)$ up to order 2 and $a, b$ up to order 1 with $a_0=a(0)$, $b_0=b(0)$ we get
$$\left |h(X) - \left (h(0) + \xi_0\cdot x' + \frac 12 (x')^T M_0 x' - \frac{a_0}{2}r^2-b_0x_nr \right) \right| \leq C|X|^3.$$
\end{proof}

In our proof above we used the following easy lemma.

\begin{lem} \label{odd}Let $\tilde H=\tilde H(\zeta, y)$ be a function defined on $\overline{B_1^+} \subset \R^2$,  which vanishes continuously on $\{\zeta =0\}$. If $\tilde H \in C^{k,\alpha}(\overline{B_1^+})$, $k \in \mathbb N, \alpha \in (0,1]$, then $\zeta^{-1}\tilde H \in C^{k-1, \alpha}(\overline{B_1^+}),$ and 
$$ \|\zeta^{-1}\tilde H \|_{C^{k-1, \alpha}}\leq \|\tilde H\|_{C^{k,\alpha}}.$$
\end{lem}

\begin{proof}
Since $\tilde H(0,y)=0$ we see that
$$\zeta^{-1}\tilde H(\zeta,y)=\int_0^1 \tilde H_\zeta(t\zeta,y)\,dt$$ and the lemma follows easily by taking derivatives in the equality above. 
\end{proof}

\section{Basic properties of a solution $g$.}

We collect here some useful  general facts about solutions $g$ to our free boundary problem \eqref{FB}, such as $C^{1/2}$-optimal regularity, asymptotic expansion near regular points of the free boundary and compactness.

First we recall some notation. Let $v \in C(B_1)$ be a non-negative function. We denote by $$B_1^+(v) := B_1 \setminus \{(x,0) : v(x,0) = 0 \} \subset \R^{n+1}$$
 and by $$F(v) := \p_{\R^n}(B_1^+(v) \cap \mathcal{B}_1)\cap \mathcal{B}_1 \subset \R^{n}.$$
Also, we denote by $P$ the half-hyperplane $$P:= \{X \in \R^{n+1} : x_n \leq 0, s=0\}.$$

Given a $C^2$ surface $\mathcal S$ in $\R^{n-1}$, we often work with functions of the form $V=V_{\mathcal S,a,b}$ (see Definition \ref{vS}). We remark that we can still apply the boundary Harnack inequality with $V$  in a neighborhood of $\mathcal S$ since in this set $V$ is comparable with a harmonic function $H$ with $F(H)=\mathcal S$.

 Indeed, after a dilation we may assume that $V=V_{\mathcal S,a,b}\in \mathcal V_\delta$, that is the curvatures of $\mathcal S$ in $B_2$ and $|a|$, $|b|$ are bounded by $\delta$ small, universal.
Let 
 $$V_1:=V_{\mathcal S, a-2n\delta,b}, \quad V_2:=V_{\mathcal S,a+2n \delta, b}$$
 and notice that $V_1$ is a supersolution and $V_2$ is a subsolution in $B_1$ (see Proposition \ref{sub}). Also  
 $$1/2 V_2 \le V \le 2 V_1\le 2 V_2,$$ hence there exists $H$ between $1/2V_2$ and $2V_1$, with $1/4 V \le H \le 4V$, $H$ harmonic in $\{H>0\}$ and $F(H)=\mathcal S$.

We obtain the following version of the boundary Harnack inequality. 

\begin{lem}\label{bhV} Let $V:=V_{\mathcal S, a, b} \in \mathcal V_{\delta_0}$, for some small $\delta_0$ universal and with $0 \in \mathcal S.$ Let $w \in C(\overline{B}_1)$ be a non-negative function which is harmonic in $B_1^+(w)$. If $B_1^+(V) \subset B_1^+(w)$ then 
$$w \geq c w(\frac 1 2 e_n) V, \quad \mbox{in $B_{1/2}$}.$$ If $B_1^+(V) \subset B_1^+(w)$ then 
$$w \leq C \|w\|_{L^\infty(B_1)} V,  \quad \mbox{in $B_{1/2}$}.$$
\end{lem}

\begin{proof}
Let  $\bar w$ be the harmonic function in $B_{3/4}^+(V)$ with boundary value $w$ on $\p B_{3/4}$ and $\bar w =0$ where $\{V=0\}.$

 If $B_1^+(V) \subset B_1^+(w)$ then in view of the observation above we can apply the boundary Harnack inequality with $V$ and conclude that $$w \geq \bar w \geq c \,\bar w (\frac 1 2 e_n) V, \quad \mbox{in $B_{1/2}$}.$$ On the other hand, $$\bar w (\frac 1 2 e_n) \geq c \inf_{B_{1/4}(\frac 3 4 e_n) \cap \p B_{3/4}} \bar w.$$ Using that $w$ and $\bar w$ coincide on $\p B_{3/4}$ together with Harnack inequality we obtain our desired estimate.

 If $B_1^+(V) \subset B_1^+(w)$ then $$w \leq \bar w \leq C \bar w (\frac 12 e_n) V, \quad \mbox{in $B_{1/2}$}.$$ On the other hand,
 $$\bar w (\frac 1 2 e_n) \leq \|\bar w \|_{L^\infty(B_{3/4})}= \|w \|_{L^\infty(B_{3/4})}$$
which yields our conclusion.

\end{proof}

An immediate consequence is the following useful lemma.

\begin{lem}\label{wv1}
Let $V=V_{\mathcal Sa,b} \in \mathcal V_{\delta_0}$ be a subsolution in $B_1$, for some small universal $\delta_0$ and with $0\in S$. If $w$ is harmonic in $B_1^+(w)$ and  $B_1^+(V)\subset B_1^+(w)$ and $$w \ge V -\eps \quad \mbox{in $B_1$},$$then $$w \ge (1-C \eps)V \quad \mbox{in $B_{1/2}$}.$$ 
\end{lem}
\begin{proof}Let $q:B_{3/4} \to \R$ be the harmonic function in $B_{3/4} \cap B^+_1(V)$ which has boundary values $q=1$ on $\p B_{3/4}$ and $q=0$ on the set where $V=0$. From our hypotheses on $w$ and the maximum principle
we obtain $$w \ge V - \eps q \quad \mbox{in} \quad B_{3/4}.$$ On the other hand by Lemma \ref{bhV}, since $B_{3/4}^+(V) = B_{3/4}^+(q)$ we have $q \le C V$ in $B_{1/2},$ which together with the inequality above implies the desired result.

\end{proof}
\begin{rem}\label{wv2}
From the proof of Lemma \ref{wv1} we see that if the hypotheses on $w$ hold only outside of the ball $B_{1/8}$, i.e $$w \ge V-\eps \quad \mbox{on} \quad B_1 \setminus B_{1/8}, \quad \quad \mbox{$w$ harmonic in} \quad B_1^+(V) \setminus B_{1/8}$$ then the conclusion holds in the shell $B_{3/4}\setminus B_{1/4}$.
\end{rem}

Next we prove optimal $C^{1/2}$ regularity for viscosity solutions.

\begin{lem}[$C^{1/2}$-Optimal regularity]\label{optimal}Assume $g$ solves \eqref{FB} in $B_1$ and $0 \in F(g)$. Then $$g(x,0) \le C |d(x)|^{1/2} \quad \mbox{in $\mathcal B_{1/2}$}$$
where $d(x)$ represents the distance from $x$ to $F(g)$. Also
$$\|g\|_{C^{1/2}(B_{1/2})} \le C (1+g(\frac 1 2 e_{n+1})).$$ 
\end{lem}

\begin{proof} 
The first assertion follows in a standard way from the free boundary condition. By scaling, we need to show that if $g$ is defined in $B_2$, $0 \in F(g)$ and $\mathcal B_1(e_n)\subset B_2^+(g)$ then $u(e_n) \le C$ for some large $C$ universal. 

By a rescaled version of Lemma \ref{bhV} and Harnack inequality we have that in a neighborhood of $0$, $$g \ge c g(e_n) V_{\mathcal S, 2n, 0}, \quad \mathcal S=\p \mathcal B_1(e_n)$$ with $V_{S,2n,0}$ a subsolution near 0. The free boundary condition gives $1 \ge c g(e_n)$ which provides a bound for $g(e_n)$.

For the second inequality we write 
$$g=g_0+g_1 \quad \mbox{in $D:=B_{3/4} \cap \{s > 0\}$},$$
with  $g_0,g_1$ harmonic in $D$ and satisfying the following boundary conditions
$$g_0 = g \quad \mbox{on $\{s=0\} \cap \p D$}, \quad g_0=0 \quad \mbox{on $\{s>0\} \cap \p D$},$$
$$g_1 = 0 \quad \mbox{on $\{s=0\} \cap \p D$}, \quad g_1=g \quad \mbox{on $\{s>0\} \cap \p D$}.$$

From our estimate for $g$ on $\{s=0\}$, we obtain
$$\|g_0\|_{C^{1/2}(B_{1/2}\cap D)} \le C\|g\|_{C^{1/2}(\mathcal B _{3/4})}\le C,$$
which together with the bound
$$\|g_1\|_{C^{1/2}(B_{1/2}\cap D)} \le C g_1(\frac 12 e_{n+1}) \le C g(\frac 12 e_{n+1}),$$
gives the desired conclusion.

\end{proof}

Next we prove that if $F(g)$ admits a tangent ball at 0 either from the positive or from the zero phase, then $g$ has an asymptotic expansion of order $o(|X|^{1/2})$. This expansion also justifies our definition of viscosity solution to the free boundary problem \eqref{FB}. We remark however that this expansion holds also for an arbitrary harmonic function $w$ which does not necessarily satisfy the free boundary condition.

\begin{lem}[Expansion at regular points from one side]\label{oneside}
Let $w \in C^{1/2}(B_1)$ be 1/2-Holder continuous, $w \ge 0$, with $w$ harmonic in $B_1^+(w)$. 
If $$0 \in F(w), \quad \mathcal{B}_{1/2}(1/2 e_n) \subset B_1^+(w),$$
then 
$$w= \alpha U + o(|X|^{1/2}), \quad \mbox{for some $\alpha>0$}.$$ 
The same conclusion holds for some $\alpha \ge 0$ if $$\mathcal{B}_{1/2}(-1/2 e_n) \subset \{w=0\}.$$
\end{lem}

\begin{proof}

We define $$\alpha:=\inf_{\nu \notin P} \liminf_{t \to 0^+} \frac {w}{U}(t \nu).$$

First we notice that $\alpha>0$. Indeed, by a rescaled version of Lemma \ref{bhV} $$w\ge c \, w(\frac 1 2 e_n) V_{\mathcal S,0,0}, \quad \mathcal S= \p \mathcal B_{\frac 12}(\frac 1 2 e_n)$$ near the origin, for some $c>0$. This implies that $\alpha \ge c \, w(\dfrac 1 2 e_n) >0$.

Assume by contradiction that the conclusion of the lemma does not hold with this choice of $\alpha$. Then there exist $\delta_1>0$ and a sequence of points $y_k \to 0$ such that
\begin{equation}\label{wyk}|w(y_k) -\alpha U(y_k)| \ge \delta_1|y_k|^{1/2}.\end{equation}
Since $w$ is 1/2-Holder continuous on $B_1$, the rescalings
$$w_k(x):=|y_k|^{-\frac 12}w(|y_k| x),$$
are uniformly 1/2 Holder continuous and
after passing to a subsequence we can assume that $w_k$ converge uniformly on compact 
sets to a limiting function $w_* \in C(\mathbb{R}^n)$.  We obtain
$$w_* \ge \alpha U, \quad \Delta w_*=0 \quad \mbox{in $\R^n \setminus P$},$$
and in view of \eqref{wyk} there exists a point $y_*$, $|y_*|=1$ such that
$$w_*(y_*) \ge \alpha U(y_*) + \delta_1.$$
Using boundary Harnack inequality we find \begin{equation}\label{w*}w_* \ge \alpha (1+ \delta_2) U \quad \mbox{in $B_1$},\end{equation}
for some $\delta_2>0$ small. Now we let
$$V=V_{\frac {\delta_2}{2n}I, 0,\delta_2,0}$$ and we notice that $V$ is subharmonic in $B_1$ (by Proposition \ref{sub}) and satisfies
$$V(X)=v_{\delta_2,0}(t,s)=(1+\frac{\delta_2}{4}\rho)U(t,s) $$ $$\le (1+ \frac{\delta_2}{2}) U(t,s) \le (1+\frac{\delta_2}{2})U(x_n,s).$$ Thus \eqref{w*} gives,
\begin{equation}\label{w*new}w_* \ge \alpha (1+\frac{\delta_2}{4}) V \quad \mbox{in $B_1.$}\end{equation}
From the existence of a tangent ball at the origin included in $\{w>0\}$ we see that for all large $k$, $w_k$ is harmonic in the set where $\{V>0\}$. Thus we conclude from \eqref{w*new} that in $B_1$
$$w_k \ge \alpha (1+\frac{\delta_2}{4})V - \eps_k, 
\quad \mbox{for some} \quad \eps_k \to 0.$$
By Lemma \ref{wv1} we find that for all large $k$, $$w_k \ge (1-C \frac{\eps_k}{\alpha} ) \alpha (1+\frac{\delta_2}{4})V \ge \alpha (1+\frac{\delta_2}{8})V \quad \mbox{ in $B_{1/2}$.}$$
 This implies that for any $\nu \notin P$
$$  \liminf_{t \to 0^+}\frac {w}{U}(t \nu)=\liminf_{t \to 0^+}\frac {w_k}{U}(t \nu) \ge \alpha (1+\frac{\delta_2}{8}),$$
which contradicts the minimality of $\alpha$.
\end{proof}

\begin{rem}\label{noc12} If we assume that $F(w)$ admits a uniform tangent ball from its $0$ side at all points in $\mathcal B_{1/2}$ then the hypothesis $w \in C^{1/2}(B_{1/4})$ is satisfied and therefore $w$ has an expansion at all points in $F(w) \cap \mathcal B_{1/4}$. Indeed, by Lemma \ref{bhV} we know that $$w \le C\|w\|_{L^\infty} V_{\mathcal \p \mathcal B_r(x_0),0,0}$$ with $\p \mathcal B_r(x_0)$ a tangent sphere to $F(w)$ from the $0$ side, and this implies
$$w(x) \le C\|w\|_{L^\infty} \,dist(x, \{w=0\})^{1/2}, \quad \quad \forall x \in B_{1/4},$$
which gives $w \in C^{1/2}(B_{1/4})$.

\end{rem}

In general, the term $o(|X|^{1/2})$ in the expansion for $w$ can be improved in $o(U)$ in the non-tangential direction to $F(w)$.
For example assume that $0 \in F(w) \in C^2$ and $e_n$ is the normal to $F(w)$ at $0$ which points towards the positive phase. Then the  non-tangential limit
$$\lim_{x \in \mathcal C, \, x \to 0} \, \, \,  \frac {w}{U}=\alpha  $$ 
where $\mathcal C \subset \R^n \setminus P$ is a cone whose closure does not contain $L=\{x_n=0, s=0\}$. 

Indeed, by Lemma \ref{oneside} and Remark \ref{noc12} we have that $w=\alpha U + o(|X|^{1/2})$. Now the limit above follows  by applying boundary Harnack inequality for $U$ and $w$ in the sets $\mathcal C_1 \cap (B_r \setminus B_{r/2})$ for all $r$ small, where $$\mathcal C_1:=\{ |x'| > \mu |(x_n,s)|\}$$ is such that $\overline {\mathcal C} \subset \mathcal C _1 \cup \{0\}$.

\begin{rem}
\label{newdef} In the definition of viscosity  solutions for our free boundary problem (see Definition 2.3) we can restrict the test functions only to the class  of subsolutions and supersolutions of the form $cV_{\mathcal S,a,b}.$ 

Precisely we say that $g$ is a solution to \eqref{FB} if 

1) $\triangle g =0$ in $B_1^+(g);$ 

2) for any point $X_0 \in F(g)$ there exists no  $V_{\mathcal S,a,b}$ such that in a neighborhood of $X_0$, $V_{\mathcal S, a, b}$ is a subsolution and 
$$ g \ge \alpha V_{S,a,b}, \quad \mbox{for some $\alpha >1$}$$ with  $\mathcal S$ touching strictly $F(g)$ at $X_0$ from the positive side.

Analogously there is no supersolution $V_{\mathcal S,a,b}$ such that
$$ g \le \alpha V_{S,a,b}, \quad \mbox{for some $\alpha <1$}$$ and $\mathcal S$ touches strictly $F(g)$ at $X_0$ from the $0$ side.

In order to prove this statement we need to show that if we can touch $g$ by below at a point $X_0\in F(g)$ with a comparison subsolution $v$ as in Definition 2.2, then we can touch also with a subsolution $\alpha V_{\mathcal S,a,b}$ as above. A similar statement holds for supersolutions. 

Assume for simplicity that $X_0=0$, $e_n$ is normal to $F(v)$ at 0 and $g \ge v$ in $\overline B_1$. Let $\bar v$ be the harmonic replacement for $v$ in $B_1^+(v)$. In view of of Remark \ref{noc12}$$\bar v=\alpha U+o(|X|^{1/2}), \quad \mbox{for some $\alpha\geq 1$.}  $$ We claim that $\alpha >1.$ Indeed,  $\bar v -v \geq 0$ is superharmonic in $B_1^+(v)$ and vanishes continuously on $\{v=0\} \cap B_1.$  If $\bar v -v \equiv 0$, then our claim follows from the definition of a comparison subsolution. Otherwise, by the boundary Harnack inequality $$\bar v - v \geq \sigma \bar v $$ in a neighborhood of the origin, for some $\sigma>0$. Thus $\bar v \geq v/(1-\sigma)$ near the origin and again the claim follows from the expansion of $v$ at the origin.   

The rescalings $v_k=r_k^{-1/2}v(r_kx)$ converge uniformly on compact sets to $\alpha U$, with $\alpha >1$. As in the proof  of Lemma \ref{oneside} we obtain that there exists $\delta$ small such that for all large $k$, $$v_k \ge V:=(1+\delta) V_{\frac {\delta} {2n}I,0,\delta,0}$$ and $F(V)$ touches strictly $F(v_k)$ at the origin from the positive side. Rescaling back we obtain the desired conclusion.

\end{rem}

Next we prove a compactness result for viscosity solutions to \eqref{FB} whose free boundaries converge in the Hausdorff distance.

\begin{prop}[Compactness]\label{compactness}
Assume $g_k$ solve \eqref{FB} and converge uniformly to $g_*$ in $B_1$, and $\{g_k =0\}$ converges in the Hausdorff distance to $\{g_*=0\}$. Then $g_*$ solves \eqref{FB} as well.
\end{prop}

\begin{proof} Clearly $g_*$ is harmonic in $B_1^+(g_*).$
In view of Remark \ref{newdef} we need to check say that if $0 \in F(g_*)$ there exists no subsolution $V_{M,0,a,b}$ such that in a neighborhood of $0$, 
$$ g_* \ge \alpha V_{M,0,a,b}, \quad \mbox{for some $\alpha>1$}$$ and $F(V)$ touches strictly $F(g_*)$ at $0$ from the positive side. A similar statement can be checked also for supersolutions.

Assume by contradiction that such a $V=V_{M,0,a,b}$ exists. Then after a dilation we may assume that $V \in \mathcal V_\delta$ for some small $\delta$ and $V$ is a subsolution in $B_1.$ 

For any $\eps>0$ there exists $\sigma>0$ such that for all $|t| \leq \sigma$ and all large $k$'s
$$W_t(X):=\alpha V(X+te_n) \leq g_k -\eps \quad \mbox{in $B_1$},$$
and
$$F(W_{-\sigma})\subset B_1^+(g_k), \quad \quad F(W_t) \setminus B_{1/8} \subset B_1^+(g_k).$$ 
By Lemma \ref{wv1} and Remark \ref{wv2} we obtain that
$$g_k \ge (1- C \eps) W_{-\sigma} \quad \mbox{in $B_{1/2}$}$$
and
$$g_k \ge (1- C \eps) W_{t} \quad \mbox{in $B_{3/4}\setminus B_{1/4}$, for all $|t| \le \sigma$.}$$
By choosing $\eps$ small (depending on $\alpha$) we see that the functions $W_t$ are strict subsolutions to our free boundary problem, and hence the inequality above can be extended in the interior (see Lemma \ref{linearcomp}) i.e.,
$$g_k \ge (1- C \eps) W_{t} \quad \mbox{in $B_{1/2}$}.$$
Writing this for $t=\sigma$ we see that $\{g_k=0\}$ stays outside a neighborhood of the origin and we contradict the convergence in the Hausdorff distance to $\{g_*=0 \}$. 

\end{proof}

We conclude this section by showing that our flatness assumption on the free boundary  $F(g)$, implies closeness of $g$ and $U$.

\begin{lem}\label{differentassumption} Assume $g$ solves \eqref{FB}. Given any $\delta>0$ there exist $\bar \eps>0$ and $\mu>0$ depending on $\delta$ such that if 
\begin{equation}\label{Flat2} \{x \in \mathcal{B}_1 : x_n \leq -\bar \eps\} \subset \{x \in \mathcal{B}_1 : g(x,0)=0\} \subset \{x \in \mathcal{B}_1 : x_n \leq \bar \eps \},\end{equation} then 
$$U(X-\mu \delta e_n) \le g(X) \le U(X+\mu \delta e_n)  \quad \mbox{in $B_\mu$.}$$
\end{lem}

\begin{proof}
The proof is by compactness. Assume by contradiction that a sequence of functions $g_k$ satisfies the hypotheses with $\bar \eps_k \to 0$ but the conclusion does not hold. 

Notice that by Harnack inequality $g_k(e_{n+1}/2)$ is bounded be a multiple of $g_k(e_n/2)$ which in view of Lemma \ref{optimal} is bounded by a universal constant. Hence by the second claim in Lemma \ref{optimal} the $g_k$'s have uniformly bounded $C^{1/2}$ norms on compact subsets of $B_1$. After passing to a subsequence we can assume that $g_k$ converges uniformly on compact sets of $B_1$ to a function $g_*$ with
\begin{equation}\label{g*}
\triangle g_*=0 \quad \mbox{in $B_{1} \setminus P$}, \quad \quad g_*=0 \quad \mbox{on $P \cap B_1$.}
\end{equation}

By Remark \ref{noc12}, $g_*$ is $C^{1/2}$. Moreover, the derivatives of $g_*$ in the $x'$ direction satisfy again \eqref{g*} and we obtain $$\|D_{x'}^\beta g_*\|_{C^{1/2}(B_{1/2})} \le C(\beta).$$
Now we can separate the variables and write
 $$\Delta_{x_n,s}g_*=-\Delta_{x'}g_*$$
and we can argue as in Theorem \ref{class} to obtain
$$|g_*(X)-\alpha U(X)| \le C |X|^{3/2}$$ with $C$ universal.

We now want to apply Proposition \ref{compactness} to conclude that  $g_*$ solves \eqref{FB} and hence $\alpha=1$. To do so, we must guarantee that $g_*>0$ in $B_{1} \setminus P.$
Otherwise $g_* \equiv 0$ and hence $\|g_k\|_{L^\infty(B_{1/2})} \to 0$. Let $\mathcal B_k:=\mathcal B_{1/8}(x_k)$ be a ball tangent to $F(g_k)$ from the zero side at some point $y_k \in B_{1/8}$. Then, since $\|g_k\|_{L^\infty(B_{1/2})} \to 0$, we have by Lemma \ref{bhV} $$g_k \leq \sigma_k V_{\p B_k,0,0}, \quad \mbox{with $\sigma_k \to 0$}.$$
This contradicts the free boundary condition for $g_k$ at $y_k$.

In conclusion, $g_*$ solves \eqref{FB} and 

$$|g_*(X)-U(X)| \le C |X|^{3/2}$$ with $C$ universal.

Rescaling we find
$$| g_{k,\mu}(X)-U(X)| \le C \mu \quad \mbox{in $B_2$}, \quad \mbox{with} \quad g_{k,\mu}(X):=\mu^{-1/2}g_k(\mu X).$$Thus $$g_{k,\mu}(X)  \ge U(X)-C\mu \ge U(X-\bar\eps_k \mu^{-1} e_n)-C \mu.$$  Now we use that $F(g_{k,\mu}) \subset \{|x_n| \le \bar\eps_k \mu^{-1} \}$ and obtain by Lemma \ref{wv1} that in $B_1$
$$ g_{k,\mu} \ge (1-C \mu) U(X-\bar\eps_k \mu^{-1} e_n) \ge U(X-(\bar\eps_k\mu^{-1}+C \mu) e_n )$$
where the last inequality follows once more from boundary Harnack  inequality. A similar inequality bounds $g_{k,\mu}$ by above. We choose $\mu$ small depending on $\delta$ and obtain that $g_{k,\mu}$ satisfies the conclusion of the theorem
$$U(X-\delta e_n) \le \mu^{-1/2}g_k(\mu X) \le U(X+\delta e_n),$$
and we reach a contradiction.
\end{proof}

\end{document}